\documentclass[a4paper,11pt]{article}

\usepackage[T1]{fontenc}
\usepackage[utf8]{inputenc}
\usepackage[english]{babel}
\usepackage{color}
\usepackage{amsmath,amssymb,amsthm}
\usepackage{ stmaryrd }
\usepackage{latexsym}
\usepackage{graphicx,epsfig}
\usepackage{float,enumerate,array}
\usepackage[left=25mm, right=25mm,
            top=30mm, bottom=30mm]{geometry}
\newtheorem{theo}{Theorem}
\newtheorem{prop}[theo]{Proposition}
\newtheorem{lem}[theo]{Lemma}

\newtheorem*{remark*}{Remark}

\newtheorem*{claim*}{Claim}

\newcommand{\set}[1]{\left\{#1\right\}}

\newcommand{\eps}{\varepsilon}

\renewcommand{\P}{\mathbb{P}}

\newcommand{\Poi}{\mathrm{Poisson}}
\newcommand{\Bin}{\mathrm{Binomial}}

\usepackage{hyperref,pifont}

\begin{document}

\title{The Ulam-Hammersley problem for multiset permutations}
\author{\textsc{Lucas Gerin}}

\maketitle

\begin{abstract}
We obtain the asymptotic behaviour of the longest increasing/non-decreasing subsequences in a random uniform multiset permutation in which each element in $\{1,\dots,n\}$ occurs $k$ times, where $k$ may depend on $n$. This generalizes the famous Ulam-Hammersley problem of the case $k=1$. The proof relies on poissonization and on a careful non-asymptotic analysis of variants of the Hammersley-Aldous-Diaconis particle system.
\end{abstract}

\section{Introduction}

A $k$-multiset permutation of size $n$ is a  word with letters in $\{1,2,\dots ,n\}$ such that each letter appears exactly $k$ times.
When this is convenient we identify a multiset permutation  $s=\left(s(1),\dots,s(kn)\right)$ and the set of points $\{(i,s(i)),\ 1\leq i\leq kn\}$. 
We introduce two partial orders over the quarter-plane $[0,\infty)^2$:
\begin{align*}
(x,y)\prec (x',y') &\text{ if } x<x'\text{ and }y<y',\\
(x,y)\preccurlyeq (x',y') &\text{ if } x<x'\text{ and }y\leq y'. 
\end{align*}
For a finite set $\mathcal{P}$ of points in the quarter-plane  we put
\begin{align*}
\mathcal{L}_{<}(\mathcal{P})&=\max\set{L; \hbox{ there exists }P_1  \prec P_2 \prec \dots \prec P_L, \text{ where each }P_i \in \mathcal{P} },\\
\mathcal{L}_{\leq}(\mathcal{P})&=\max\set{L; \hbox{ there exists }P_1\preccurlyeq P_2 \preccurlyeq\dots \preccurlyeq P_L, \text{ where each }P_i \in \mathcal{P} }.
\end{align*}
In words the integer $\mathcal{L}_{<}(\mathcal{P})$ (resp. $\mathcal{L}_{\leq}(\mathcal{P})$) is the length of the longest increasing (resp.  non-decreasing) subsequence of $\mathcal{P}$.

\begin{figure}
\begin{center}
\includegraphics[width=12cm]{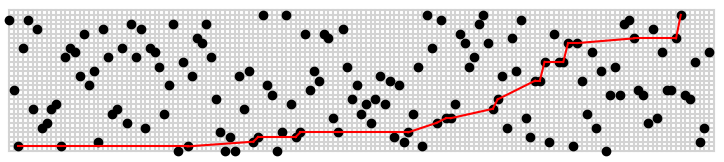}
\end{center}
\caption{A uniform $5$-multiset permutation $S_{5;30}$ of size $n=30$ and one of its longest non-decreasing subsequences.}
\end{figure}

Let $S_{k;n}$ be a $k$-multiset permutation of size $n$ drawn uniformly among the $\frac{(kn)!}{k!^n}$ possibilities. 
In the case $k=1$ the word $S_{1;n}$ is simply a uniform permutation and estimating $\mathcal{L}_{<}(S_{1;n})=\mathcal{L}_{\leq}(S_{1;n})$ is known as the Hammersley or Ulam-Hammersley problem. The first order was solved by Ver{\v{s}}ik and Kerov \cite{VersikKerov} and simultaneously by Logan and Shepp:
$$
\mathbb{E}[\mathcal{L}_{<}(S_{1;n})]  \stackrel{n\to +\infty}\sim 2\sqrt{n}.
$$
Note that the above limit  also holds in probability\footnote{We use the shortcut $\mathrm{o}_{\mathbb{P}}(a_n)$ for a sequence of random variables $(X_n)$ such that $X_n/a_n \to 0$ in probability.}: $
\mathcal{L}_{<}(S_{1;n})=  2\sqrt{n}+\mathrm{o}_{\mathbb{P}}(\sqrt{n})$. 
This problem has a long history and has revealed deep and unexpected connections between combinatorics, interacting particle systems, calculus of variations, random matrix theory, representation theory. We refer to Romik \cite{Romik} for a very nice description of this problem and some of its ramifications.

In the context of card guessing games it is asked in \cite[Question 4.3]{ContinuouslyIncreasing} the behaviour of $\mathcal{L}_{<}(S_{k;n})$ for a fixed $k$. 
Using the  Ver{\v{s}}ik-Kerov Theorem we can make an educated guess. The intuition is that, for fixed $k$, it is quite unlikely that many points at the same height contribute to the same longest increasing/non-decreasing subsequence. Thus at the first order everything should happen as if the $kn$ points had distinct heights and we expect that
$$
\mathcal{L}_{<}(S_{k;n})\approx \mathcal{L}_{\leq}(S_{k;n}) \approx  \mathcal{L}_{<}(S_{1;kn}) \approx 2\sqrt{kn}.
$$
The original motivation of the present paper was to make this approximation rigorous. We actually adress this question in the case where $k$ depends on $n$.

\begin{theo}[Longest increasing subsequences]\label{Th:Strict}
Let $(k_n)$ be a sequence of integers such that $k_n\leq n$ for all $n$.
Then\footnote{If $k_n\geq n$ for some $n$ then the following greedy strategy shows that $\mathbb{E}[\mathcal{L}_{<}(S_{k_n;n})]= n-\mathrm{o}(n)$ so the picture is complete. \\
Indeed, first choose the leftmost point $(x_1,1)$ in $S_{k_n;n}$ with height $1$. Then recursively define $(x_\ell,\ell)$ at the leftmost point (if any) in $S_{k_n;n}$ with height $\ell$ such that $x_\ell >x_{\ell-1}$, and so on until you are stuck (either because $\ell=n$ or because there is no point in $S_{k_n;n}\cap (x_{\ell-1},kn]\times\set{\ell}$).
A few elementary computations show that this strategy defines an increasing path of length $n-\mathrm{o}(n)$ with probability tending to one. As $\mathcal{L}_{<}(S_{k_n;n}) \leq n$ a.s. this yields $\mathbb{E}[\mathcal{L}_{<}(S_{k_n;n})]= n-\mathrm{o}(n)$.
}
\begin{equation}\label{eq:Theta<}
\mathbb{E}[\mathcal{L}_{<}(S_{k_n;n})]=2\sqrt{nk_n}-k_n+o(\sqrt{nk_n}).
\end{equation}
\end{theo}
(Of course if $k_n=o(n)$ then the RHS of  \eqref{eq:Theta<} reduces to  $2\sqrt{nk_n}+\mathrm{o}(\sqrt{nk_n})$.)

\begin{theo}[Longest non-decreasing subsequences]\label{Th:Large}
Let $(k_n)$ be an arbitrary sequence of integers.
Then
\begin{equation}\label{eq:Theta_leq}
\mathbb{E}[\mathcal{L}_{\leq}(S_{k_n;n})]=2\sqrt{nk_n}+k_n + o(\sqrt{nk_n}).
\end{equation}
\end{theo}

\paragraph{Strategy of proof and organization of the paper.} 
In Section \ref{Sec:Small} we first provide the proof of Theorems \ref{Th:Strict} and \ref{Th:Large} in the case of a constant or slowly growing sequence $(k_n)$. The proof is elementary (assuming known the Ver{\v{s}}ik-Kerov Theorem).

For the general case we first borrow a few tools in the literature. In particular we introduce and analyze poissonized versions of  $\mathcal{L}_{<}(S_{k_n;n}),\mathcal{L}_{\leq}(S_{k_n;n})$. As already suggested by Hammersley (\cite{Hamm}, Sec.9) and achieved by Aldous-Diaconis \cite{AldousDiaconis} the case $k=1$ can be tackled by considering an interacting particle system which is now known  as the Hammersley or Hammersley-Aldous-Diaconis (HAD) process. 

In Section \ref{Sec:Hammersley} we introduce and analyze the two variants of the Hammersley process adapted to multiset permutations. The first one is the discrete-time HAD process \cite{Ferrari,FerrariMartin}, the second one had recently appeared in \cite{TheseBoyer} with a connection to the O'Connell-Yor Brownian polymer. The standard path to analyze Hammersley-like processes consists in using subadditivity to prove the existence of a limiting shape and then proving that this limiting shape satisfies a variational problem. Typically this variational problem is solved either using convex duality  \cite{Sepp,CiechGeorgiou} or through the analysis of \emph{second class particles} \cite{CatorGroeneboom2,CiechGeorgiou}. 
The issue here is that since we allow $k_n$ to have different scales we cannot use this approach and we need to derive non-asymptotic bounds for both processes. This is the purpose of Theorem \ref{Th:BorneInf_L} whose proof  is the most technical part of the paper. 
In Section \ref{Sec:Dep} we detail the multivariate de-poissonization procedure in order to conclude the proof of Theorem \ref{Th:Strict}. De-poissonization is more convoluted for non-decreasing subsequences: see  Section \ref{Sec:Dep2}.

\paragraph{Beyond expectation.} 
In the course of the proof we actually obtain results beyond the estimation of the expectation. We obtain concentration inequalities for the poissonized version of $\mathcal{L}_{<}(S_{k_n;n}),\mathcal{L}_{\leq}(S_{k_n;n})$: see Theorem \ref{Th:BorneInf_L} and also the discussion in Section \ref{Sec:Deviations}.
We also obtain the convergence in probability, unfortunately for some technical reasons we miss a small range of scales of $(k_n)$'s.
\begin{prop}\label{prop:cvProba}
Let $(k_n)$ be either a \emph{small} or a \emph{large} sequence.  Then
$$
\frac{\mathcal{L}_{<}(S_{k_n;n})}{2\sqrt{nk_n}-k_n}\stackrel{\text{prob.}}\to 1,\qquad 
\frac{\mathcal{L}_{\leq}(S_{k_n;n})}{2\sqrt{nk_n}+k_n}\stackrel{\text{prob.}}\to 1.
$$
\end{prop}
We refer to \eqref{eq:condition_small},\eqref{eq:condition_large} below for the formal definitions of small/large sequences. Let us just say that sequences such that  $k_n=\mathcal{O}((\log n)^{1-\eps})$ for some $\eps>0$ are small while sequences such that $(\log n)^{1+\eps}=\mathcal{O}(k_n)$ are large. 
Sequences  in-between are neither small nor large so in Proposition \ref{prop:cvProba} we miss scales like $k_n\approx \log(n)$.

Regarding fluctuations a famous result by Baik, Deift and Johansson \cite[Th.1.1]{BDK} states that
$$
\frac{\mathcal{L}_{\leq}(S_{1;n})-2\sqrt{n}}{n^{1/6}}\stackrel{(d)}{\to} \mathrm{TW}
$$
where $\mathrm{TW}$ is the Tracy-Widom distribution. The intuition given by the comparison with the Hammersley process would suggest that the fluctuations of $\mathcal{L}_{<}(S_{k_n;n})$, $\mathcal{L}_{\leq}(S_{k_n;n})$ might be of order $(k_n n)^{1/6}$ as long as $(k_n)$ does not grow too fast. 
A natural question to explore for furthering this work would involve understanding for which $(k_n)$ the model preserves KPZ scaling exponents. The non-asymptotic estimates of Section \ref{Sec:Hammersley} could serve as a first step in this direction.

\paragraph{Comparison with previous works.} There are only few random sets $\mathcal{P}$ for which the asymptotics of $\mathcal{L}_{<}(\mathcal{P}),\mathcal{L}_{\leq}(\mathcal{P})$ are known: 
\begin{enumerate}
\item[-] As already mentioned, the case of a uniform permutation (and its poissonized version) is very well understood, via different approaches. 
For proofs close to the spirit of the present paper, we refer to \cite{AldousDiaconis} and \cite{CatorGroeneboom}.
\item[-] The case where $\mathcal{P}$ is given by a field of i.i.d. Bernoulli random variables on the square grid has been solved by Sepp\"al\"ainen in \cite{Sepp} for  $\mathcal{L}_{<}$ and in \cite{Sepp2}  for  $\mathcal{L}_{\leq}$. (See \cite{NousAlea} also for an elementary proof of both results).
\end{enumerate}
We are not aware of previous results for multiset permutations.
However Theorems \ref{Th:Strict} and \ref{Th:Large} in the linear regime $k_n\sim \mathrm{constant}\times n$ should be compared to a result by Biane (\cite[Theorem 3]{Biane}). \\
We need a few notations to describe his result.
Let $\mathcal{W}_{q_N;N}$ be the random word given by of $q_N$  i.i.d. uniform letters in $\{1,2,\dots, N\}$.  
The word $\mathcal{W}_{q_N;N}$ is not a multiset permutation but since for large $N$ there are in average $q_N/N$ points on each horizontal line of $\mathcal{W}_{q_N;N}$ we expect that $\mathcal{L}_<(\mathcal{W}_{q_N;N}) \approx \mathcal{L}_<(S_{q_N/N;N})$ and $\mathcal{L}_\leq(\mathcal{W}_{q_N;N})\approx \mathcal{L}_\leq(S_{q_N/N;N}) $.

Biane obtains the exact limiting shape of the random Young Tableau induced through the RSK correspondence by $\mathcal{W}_{q_N;N}$  in the regime where $\sqrt{q_N}/N \to c$ for some constant $c>0$. 
As the length of the first row (resp. the number of rows) in the  Young Tableau corresponds to the length of the longest non-decreasing subsequence in $\mathcal{W}_{k;n}$ (resp. the length of the longest decreasing sequence) a consequence of (\cite[Theorem 3]{Biane}) is that, in probability,
$$
\liminf \frac{1}{\sqrt{q_N}} \mathcal{L}_<(\mathcal{W}_{q_N;N}) \geq (2-c),\qquad
\limsup \frac{1}{\sqrt{q_N}} \mathcal{L}_\leq(\mathcal{W}_{q_N;N}) \leq (2+c).
$$

For that regime our Theorems \ref{Th:Strict} and \ref{Th:Large} respectively suggest:
\begin{align*}
\mathcal{L}_<(\mathcal{W}_{q_N;N}) \approx \mathcal{L}_<(S_{q_N/N;N}) \approx \mathcal{L}_<(S_{c^2 N;N}) \sim  2Nc-c^2N\sim (2-c)\sqrt{q_N},\\ 
\mathcal{L}_\leq(\mathcal{W}_{q_N;N})\approx \mathcal{L}_\leq(S_{q_N/N;N})  \approx \mathcal{L}_\leq(S_{c^2 N;N}) \sim 2Nc+c^2N\sim (2+c)\sqrt{q_N},
\end{align*}
which is indeed consistent with Biane's result.
\section{Preliminaries: the case of small $k_n$}\label{Sec:Small}

We first prove Theorems \ref{Th:Strict} and \ref{Th:Large} in the case of a \emph{small} sequence $(k_n)$.
We say that a sequence $(k_n)$ of integers is \emph{small} if 
\begin{equation}\label{eq:condition_small}
k_n^2(k_n)!=\mathrm{o}(\sqrt{n}).
\end{equation}
Note that a sequence of the form $k_n=(\log n)^{1-\eps}$ is small while $k_n=\log n$ is not small.

\begin{proof}[Proof of Theorems \ref{Th:Strict} and \ref{Th:Large} in the case of a small sequence $(k_n)$]
\noindent (In order to lighten notation we skip the dependence in $n$ and write $k=k_n$.)\\
Let $\sigma_{kn}$ be a random uniform permutation of size $kn$. We can associate to $\sigma_{kn}$ a $k$-multiset permutation $S_{k;n}$ in the following way. For every $1\leq i\leq kn$ we put
$$
S_{k;n}(i)=\lceil \sigma(i)/k \rceil.
$$
It is clear that $S_{k;n}$ is uniform and we have
\begin{equation}\label{eq:EncadrementVersik}
\mathcal{L}_{<}(S_{k;n}) \leq \mathcal{L}_{\leq}(\sigma_{kn})\leq \mathcal{L}_{\leq}(S_{k;n}).
\end{equation}
The Ver{\v{s}}ik-Kerov Theorem says that the middle term in the above inequality grows like $2\sqrt{kn}$. Hence we need to show that if $(k_n)$ is small then
$$
 \mathcal{L}_{\leq}(S_{k;n})= \mathcal{L}_{<}(S_{k;n}) +o_\mathbb{P}(\sqrt{kn}),
$$
which proves the small case of Proposition \ref{prop:cvProba} and Theorems \ref{Th:Strict} and \ref{Th:Large}.
For this purpose we introduce for every $\delta > 0$ the event
$$
\mathcal{E}_\delta := \left\{\mathcal{L}_{\leq}(S_{k;n})\geq \mathcal{L}_{<}(S_{k;n}) +\delta \sqrt{n}\right\}.
$$
If $\mathcal{E}_\delta$ occurs then in particular there exists a non-decreasing subsequence with $\delta \sqrt{n}$ ties, \emph{i.e.} points of $S_{k;n}$  which are at the same height as their predecessor in the subsequence. These ties have distinct heights $1\leq i_1<\dots < i_\ell \leq n$ for some $\delta \sqrt{n}/k \leq \ell \leq \delta \sqrt{n}$.
Fix
\begin{itemize}
\item Integers $m_1,\dots ,m_\ell \geq 2$ such that $(m_1-1)+\dots +(m_\ell -1) = \delta \sqrt{n}$ ;
\item Column indices $r_{1,1}<\dots <r_{1,m_1}<r_{2,1}< r_{2,m_1}< \dots< r_{\ell,1} < \dots <r_{1,m_\ell}$.
\end{itemize}
We then introduce the event
\begin{multline*}
F=F\left((i_\ell)_\ell, (r_{i,j})_{i\leq \ell,j\leq m_i}\right) \\= \left\{ S(r_{1,1})=\dots = S(r_{1,m_1})=i_1,S(r_{2,1})=\dots = S(r_{2,m_1})=i_2
,\dots,S(r_{\ell,1})=\dots = S(r_{1,m_\ell})=i_\ell\right\}.
\end{multline*}
By the union bound (we skip the integer parts)
$$
\mathbb{P}(\mathcal{E}_\delta )\leq  \sum_{\delta\sqrt{n}/k \leq \ell \leq  \delta\sqrt{n}}\ \ 
\sum_{1\leq i_1<\dots \leq i_\ell \leq n}\ \ 
\sum_{ (r_{i,j})_{i\leq \ell,j\leq m_i}} \mathbb{P}\left(F\left((i_\ell)_\ell, (r_{i,j})_{i\leq \ell,j\leq m_i}\right)\right).
$$
Using that
$$
\mathrm{card}\left\{\sum m_i=\delta\sqrt{n}+\ell; \text{ each }m_i\geq 2\right\}=\mathrm{card}\left\{\sum p_i=\delta\sqrt{n}; \text{ each }p_i\geq 1\right\}=\binom{\delta\sqrt{n}-1}{\ell-1}
$$
we obtain
\begin{align*}
\sum_{(r_{i,j})_{i\leq \ell,j\leq m_i}}\mathbb{P}(F)
&=\frac{1}{\binom{kn}{k\ k\ \dots\ k}}
\underbrace{\binom{nk}{\sum m_i}}_{\text{choices of $r$'s}}\ 
\underbrace{\binom{\delta\sqrt{n}-1}{\ell-1}}_{\text{choices of }m_i's}\ 
\underbrace{\binom{kn-\sum m_i}{(k-m_1)\ (k-m_2)\ \dots (k-m_\ell)  k \dots k}}_{\text{choices of $kn-\sum m_i$ remaining points}}
\\
&=
 \frac{(k!)^\ell(\delta\sqrt{n}-1)!}{(\delta\sqrt{n}+\ell)!(\delta\sqrt{n}-\ell)!(\ell-1)!(k-m_1)!(k-m_2)!\times \dots \times (k-m_\ell)!}.
\end{align*}
Bounding each factor $(k-m_i)!$ by $1$ we get
$$
\sum_{(r_{i,j})_{i\leq \ell,j\leq m_i}}\mathbb{P}(F)\leq\frac{(k!)^\ell}{(\delta\sqrt{n})^{\ell +1}(\delta\sqrt{n}-\ell)!(\ell-1)!}.
$$

\begin{figure}
\begin{center}
\includegraphics[width=12cm]{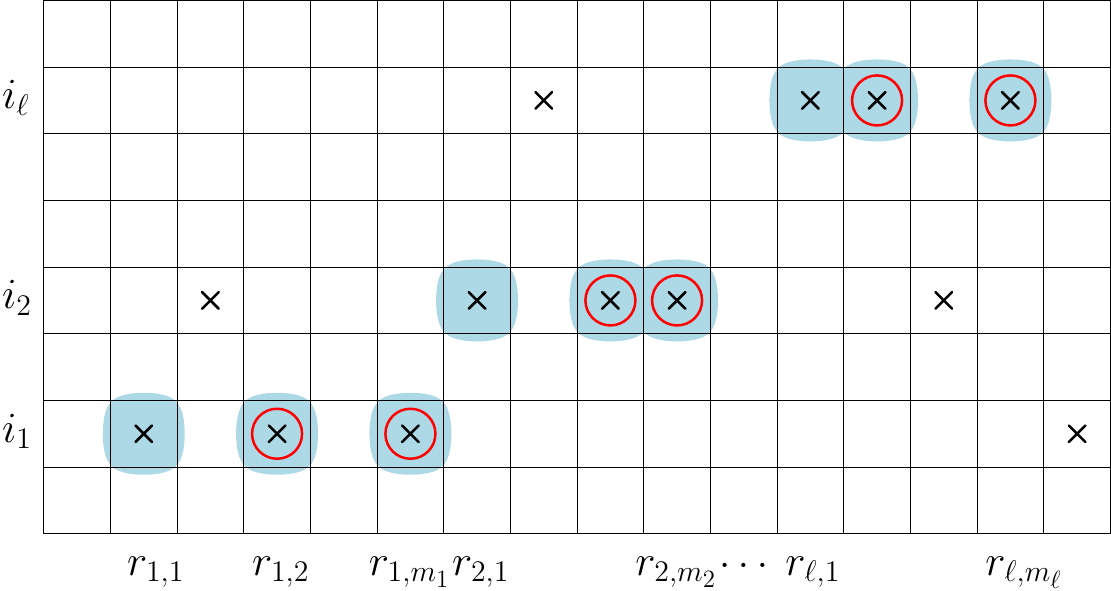}
\end{center}
\caption{The event $F$. (Ties are surrounded in red. Points with blue background represent the subsequence with $\delta\sqrt{n}$ ties.)}
\end{figure}

We now sum over $1\leq i_1<\dots \leq i_\ell \leq n$ and then sum over $\ell$:
\begin{align}
\mathbb{P}(\mathcal{E}_\delta)
&\leq   \sum_{\ell= \delta\sqrt{n}/k}^{\delta\sqrt{n}} \binom{n}{\ell}   \frac{(k!)^\ell}{(\delta\sqrt{n})^{\ell+1} (\delta\sqrt{n}-\ell)!(\ell-1)!} \notag\\
&\leq   \sum_{\ell= \delta\sqrt{n}/k}^{\delta\sqrt{n}-3} \binom{n}{\ell}   \frac{(k!)^\ell}{(\delta\sqrt{n})^{\ell+1} (\delta\sqrt{n}-\ell)!(\ell-1)!} 
+ 3   \binom{n}{\delta\sqrt{n}} \frac{(k!)^{\delta\sqrt{n}}}{(\delta\sqrt{n})^{\delta\sqrt{n}-2} (\delta\sqrt{n}-3)!} \label{eq:SommeCoupeeEnDeux}
\end{align}
Using the two following  inequalities valid for every $j\leq m$ (see \emph{e.g.} \cite[eq.(C.5)]{Cormen})
$$
 \binom{m}{j} \leq \left( \frac{me}{j}\right)^{j},\qquad m! \geq m^m\exp(-m)
$$
we first obtain that if $k_n!=\mathrm{o}(\sqrt{n})$ (which is the case if $(k_n)$ is small) then the last term of \eqref{eq:SommeCoupeeEnDeux} tends to zero. Regarding the sum we write
\begin{align}
\mathbb{P}(\mathcal{E}_\delta)&\leq \sum_{\ell= \delta\sqrt{n}/k}^{\delta\sqrt{n}-3}
\left(\frac{ne}{\ell}\right)^\ell   \frac{(k!)^\ell}{(\delta\sqrt{n})^{\ell+1}(\delta\sqrt{n}-\ell)^{\delta\sqrt{n}-\ell}e^{-\delta\sqrt{n}+\ell} (\ell-1)^{\ell-1}e^{-\ell+1}} + \mathrm{o}(1)\notag\\
&\leq \sum_{\ell= \delta\sqrt{n}/k}^{\delta\sqrt{n}-3}
\left(\frac{nek!(\delta\sqrt{n}-\ell)}{\delta\sqrt{n}\ell(\ell-1)}\right)^\ell \underbrace{\frac{(\ell-1)e^{-1}}{\delta\sqrt{n}}}_{\leq 1}   \bigg(\underbrace{\frac{e}{\delta\sqrt{n}-\ell}}_{\leq e/3<1}\bigg)^{\delta\sqrt{n}}+ \mathrm{o}(1)\notag\\
&\leq \sum_{\ell= \delta\sqrt{n}/k}^{\delta\sqrt{n}-3}
\left(\frac{\sqrt{n}ek!(\delta\sqrt{n}-\ell)}{\delta\ell(\ell-1)}\right)^\ell    \left(\frac{e}{\delta\sqrt{n}-\ell}\right)^{\ell}+ \mathrm{o}(1)\notag\\
&\leq \sum_{\ell= \delta\sqrt{n}/k}^{\delta\sqrt{n}-3}
\left(\frac{\sqrt{n}e^2k!}{\delta\ell(\ell-1)}\right)^\ell + \mathrm{o}(1)
 \leq \sum_{\ell= \delta\sqrt{n}/k}^{\delta\sqrt{n}-3}
\left(\frac{e^2k^2 k!}{\delta^3 \sqrt{n}}\right)^\ell + \mathrm{o}(1)\label{eq:Bound_oP}
\end{align}
which tends to zero for every $\delta >0$, as long as $(k_n)$ satisfies \eqref{eq:condition_small}. This proves that $ \mathcal{L}_{\leq}(S_{k;n})= \mathcal{L}_{<}(S_{k;n}) +o_\mathbb{P}(\sqrt{kn})$. Combining this with \eqref{eq:EncadrementVersik}, this proves that
$$
\frac{\mathcal{L}_{<}(S_{k_n;n})}{2\sqrt{nk_n}}\stackrel{\text{prob.}}\to 1,\qquad 
\frac{\mathcal{L}_{\leq}(S_{k_n;n})}{2\sqrt{nk_n}}\stackrel{\text{prob.}}\to 1,
$$
which is the "small" case of Proposition \ref{prop:cvProba} since $k_n= \mathrm{o}(\sqrt{nk_n})$.

To conclude the proof of small cases of Theorems \ref{Th:Strict} and \ref{Th:Large} we observe that we have the crude bounds  $ \mathcal{L}_{<}(S_{k;n})\leq n$ and $ \mathcal{L}_{\leq}(S_{k;n})\leq nk_n$. This allows us to write 
$$
\mathbb{E}\left[\big|\mathcal{L}_{\leq}(S_{k;n})- \mathcal{L}_{<}(S_{k;n})\big|\right]
\leq \delta\sqrt{n} +  nk_n\times \mathbb{P}(\text{not }\mathcal{E}_\delta)
$$
Together with  eq.\eqref{eq:Bound_oP} this implies that 
$$
\mathbb{E}[\mathcal{L}_{\leq}(S_{k;n})]= \mathbb{E}[\mathcal{L}_{<}(S_{k;n})] +o(\sqrt{nk_n}).
$$
We use again Ver{\v{s}}ik-Kerov and \eqref{eq:EncadrementVersik} to deduce that both sides are $2\sqrt{nk_n}+\mathrm{o}(\sqrt{nk_n})$.
\end{proof}

%



\section{Poissonization: variants of the Hammersley process}\label{Sec:Hammersley}

In this section we define formally and analyze two \emph{semi-discrete} variants of the Hammersley process.

\begin{remark*}
In the sequel, $\Poi(\mu)$ (resp. $\Bin(n,q)$) stand for generic random variables with Poisson distribution with mean $\mu$ (resp. Binomial distribution with parameters $n,q$).\\
Notation $\mathrm{Geometric}_{\geq 0}(1-\beta)$ stands for a geometric random variable with the convention $\mathbb{P}(\mathrm{Geometric}_{\geq 0}(1-\beta)=k)=(1-\beta)\beta^k$  for $k\geq 0$. In particular $\mathbb{E}[\mathrm{Geometric}_{\geq 0}(1-\beta)]=\frac{\beta}{1-\beta}$.
\end{remark*}

\subsection{Definitions of the processes  $L_<(t)$ and $L_\leq(t)$}

For a parameter $\lambda>0$ let $\Pi^{(\lambda)}$ be the random set $\Pi^{(\lambda)}=\cup_i \Pi_i^{(\lambda)}$ where $\Pi^{(\lambda)}_i$'s are independent  and each $\Pi^{(\lambda)}_i$ is a homogeneous Poisson Point Process (PPP) with intensity $\lambda$  on $(0,\infty)\times\{i\}$. 
For simplicity set
$$
\Pi^{(\lambda)}_{x,t}=\Pi^{(\lambda)} \cap \left( [0,x]\times\{1,\dots ,t\}\right).
$$
The goal of the present section is to obtain non-asymptotic bounds for $\mathcal{L}_<\left(\Pi^{(\lambda)}_{x,t}\right)$ and $\mathcal{L}_\leq \left( \Pi^{(\lambda)}_{x,t}\right)$. 
Indeed if we then choose
$$
\lambda_n \approx \frac{1}{n},\qquad x= nk_n,\qquad t=n
$$
then there are  $nk_n+\mathcal{O}(\sqrt{k_n})$ points on each line of a $\Pi^{(\lambda)}_{x,t}$ and we expect that
$$
\mathcal{L}_<\left( \Pi^{(\lambda_n)}_{kn,n}\right)
\approx \mathcal{L}_<(S_{k;n}),\qquad 
\mathcal{L}_\leq\left( \Pi^{(\lambda_n)}_{kn,n}\right)
\approx \mathcal{L}_\leq(S_{k;n})
$$
Fix $x>1$ throughout the section. 
For every $t\in \{0,1,2,\dots \}$ the function $y\in[0,x] \mapsto \mathcal{L}_<(y,t)$ (resp.  $\mathcal{L}_\leq(y,t)$) is a non-decreasing integer-valued function whose all steps are equal to $+1$. Therefore this function is completely determined by the finite set
$$
L_<(t):=\set{y\leq x, \mathcal{L}_<(y,t)=\mathcal{L}_<(y^-,t)+1}.
$$
(Respectively:
$$
L_\leq(t):=\set{y\leq x, \mathcal{L}_\leq(y,t)=\mathcal{L}_\leq(y^-,t)+1}.)
$$
Sets $L_<(t)$ and $L_\leq(t)$ are finite subsets of $[0,x]$ whose elements are considered as particles.
It is easy to see that for fixed $x>0$ both processes $(L_<(t))_t$ and $(L_\leq(t))_t$  are Markov processes taking their values in the family of point processes of $[0,x]$.

\begin{figure}
\begin{center}
\includegraphics[width=78mm]{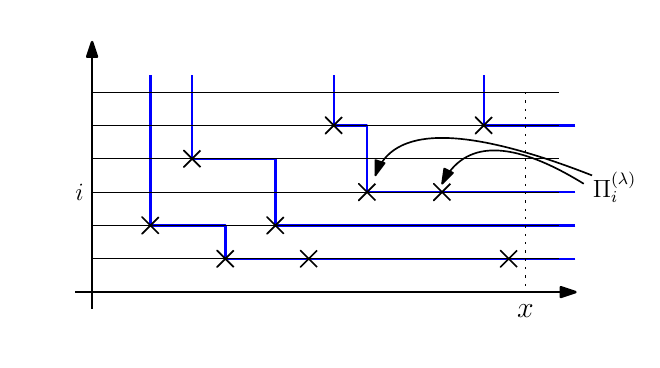} 
\includegraphics[width=78mm]{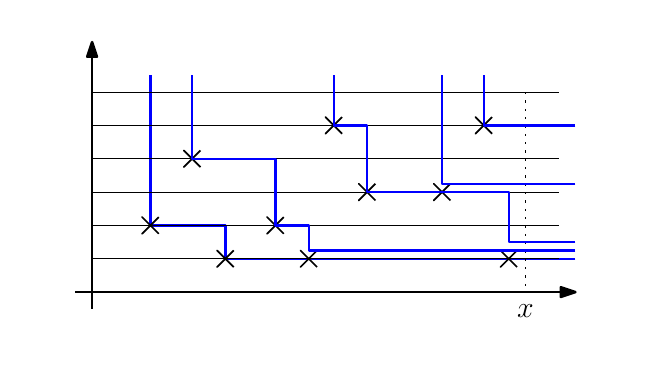} 
\includegraphics[width=78mm]{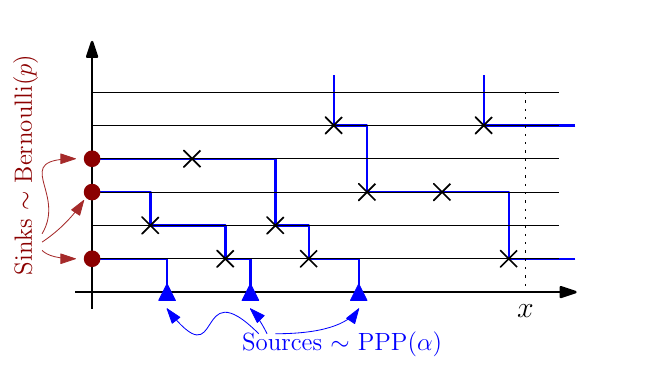} 
\includegraphics[width=78mm]{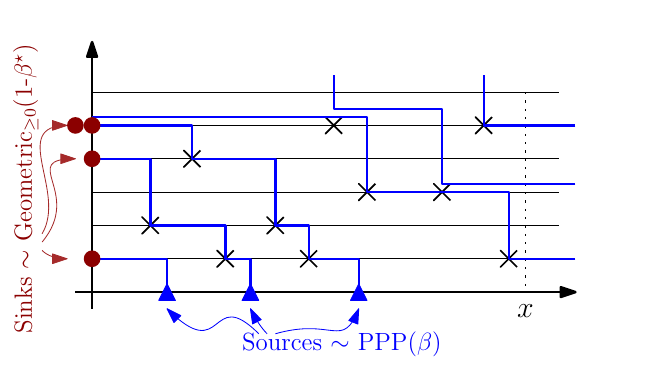} 
\end{center}
\caption{Our four variants of the Hammersley process (time goes from bottom to top, trajectories of particules are indicated in blue).
Top left: The process $L_<(t)$. Top right: The process $L_\leq(t)$. Bottom left: The process $L^{(\alpha,p)}_{< }(t)$. Bottom right: The process $L^{(\beta,\beta^\star)}_{\leq }(t)$. }
\label{fig:4schemasHammersley}
\end{figure}

Exactly the same way as for the classical Hammersley process (\cite[Sec.9]{Hamm}, \cite{AldousDiaconis}) the individual dynamic of particles is very easy to describe:
\begin{itemize}
\item {\bf The process $L_<$\ .} We put $L_<(0)=\emptyset$. In order to define $L_<(t+1)$ from $L_<(t)$ we consider particles from left to right. A particle at $y$ in $L_<(t)$ moves at time $t+1$ at the location of the leftmost available point $z$ in $\Pi_{t+1}^{(\lambda)}\cap(0,y)$ (if any, otherwise it stays at $y$). This point $z$ is not available anymore for subsequent particles, as well as every other point of $\Pi_{t+1}^{(\lambda)}\cap (0,y)$. \\
\begin{center}
\includegraphics[width=95mm]{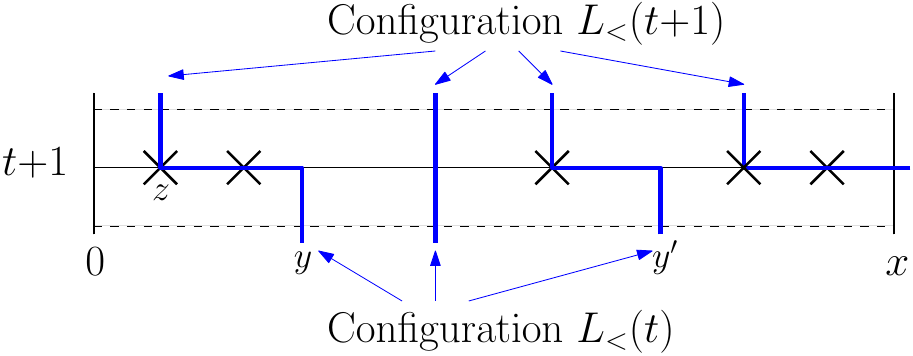} 
\end{center}
If there is a point in $\Pi_{t+1}^{(\lambda)}$ which is on the right of $y':=  \max \{ L_<(t)\}$ then a new particle is created in $L_<(t+1)$, located at the leftmost point in $\Pi_{t+1}^{(\lambda)}\cap (y',x)$. (In pictures this new particle comes from the right.)\\
A realization of $L_<$ is shown on top-left of Fig.\ref{fig:4schemasHammersley}.
\end{itemize}

\begin{itemize}
\item {\bf The process $L_\leq$\ .} We put $L_\leq(0)=\emptyset$. In order to define $L_\leq(t+1)$ from $L_\leq(t)$ we also consider particles from left to right. A particle at $y$ in $L_\leq(t)$ moves at time $t+1$ at the location of the leftmost available point $z$ in $\Pi_{t+1}^{(\lambda)}\cap(0,y)$. This point $z$ is not available anymore for subsequent particles, {\bf other points in $(z,y)$ remain available}.\\
\begin{center}
\includegraphics[width=95mm]{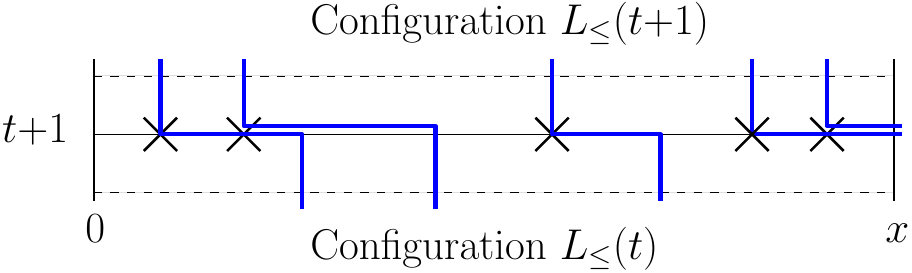} 
\end{center}
If there is a point in $\Pi_{t+1}^{(\lambda)}$ which is on the right of $y':=  \max\{ L_<(t)\}$ then new particles are created in $L_<(t+1)$, one for each point in $\Pi_{t+1}^{(\lambda)}\cap (y',x)$.\\
A realization of $L_\leq$ is shown on top-right of Fig.\ref{fig:4schemasHammersley}.
\end{itemize}
Processes $L_{< }(t)$ and $L_{\leq}(t)$ are designed in such a way that they record the length of longest increasing/non-decreasing paths in $\Pi$. 
In fact particles trajectories correspond to the level sets of the functions $(x,t)\mapsto \mathcal{L}_<\left( \Pi^{(\lambda)}_{x,t}\right)$,
$(x,t)\mapsto \mathcal{L}_\leq \left( \Pi^{(\lambda)}_{x,t}\right)$.

\begin{prop}\label{prop:NombreLignes}
For every $x$,
$$
\mathcal{L}_<\left( \Pi^{(\lambda)}_{x,t}\right) = \mathrm{card}(L_<(t)),\qquad
\mathcal{L}_\leq\left( \Pi^{(\lambda)}_{x,t}\right) = \mathrm{card}(L_\leq(t)),
$$
where on each right-hand side we consider the particle system on $[0,x]$.
\end{prop}

\begin{proof} We are merely restating the original construction from Hammersley (\cite{Hamm}, Sec.9). We only do the case of $L_<(t)$.

Let us call each particle trajectory a \emph{Hammersley line}. 
By construction each Hammersley line is a broken line starting from the right of the box $[0,x]\times [0,t]$ and is formed by a succession of north/west line segments.
Because of this, two distinct points in a given longest increasing subsequence of  $\Pi^{(\lambda)}_{x,t}$ cannot belong to the same Hammersley line. Since there are $L_<(t)$ Hammersley's lines this gives $\mathcal{L}_<\left( \Pi^{(\lambda)}_{x,t}\right) \leq \mathrm{card}(L_<(t))$.

In order to prove the converse inequality we build from this graphical construction a longest increcreasing subsequence of $ \Pi^{(\lambda)}_{x,t}$ with exactly one point on each Hammersley line. To do so, we order Hammersley's lines from bottom-left to top-right, and we build our path starting from the top-right corner. We first choose any point of  $\Pi^{(\lambda)}_{x,t}$ belonging to the last Hammersley line. We then proceed by induction: we choose the next point among the points of of  $\Pi^{(\lambda)}_{x,t}$ lying on the previous Hammersley line such that the subsequence remains increasing. (This is possible since Hammersley's lines only have North/West line segments.) This proves $\mathcal{L}_<\left( \Pi^{(\lambda)}_{x,t}\right) \geq \mathrm{card}(L_<(t))$.
\end{proof}

\subsection{Sources and sinks: stationarity}

Proposition \ref{prop:NombreLignes} tells us that in on our way to prove Theorem \ref{Th:Strict} and Theorem \ref{Th:Large} we need to understand the asymptotic behaviour of processes $L_{< },L_{\leq}$. 

It is proved in \cite{FerrariMartin} that the homogeneous PPP with intensity $\alpha$ on $\mathbb{R}$ is stationary for  $(L_<(t))_t$.
However we need non-asymptotic estimates for $(L_<(t))_t$ (and  $(L_\leq(t))_t$) on a given interval $(0,x)$.
 To solve this issue we use the trick of \emph{sources/sinks}  introduced formally and exploited by Cator and Groeneboom \cite{CatorGroeneboom} for the continuous HAD process:
\begin{itemize}
\item \emph{Sources} form a finite subset of $[0,x]\times \{0\}$ which plays the role of the initial configuration $L_<(0),L_\leq(0)$.  
\item \emph{Sinks} are points of $\{0\}\times  [1,t] $  which add up to $\Pi^{(\lambda)}$ when one defines the dynamics of $L_<(t),L_\leq(t)$. For $L_\leq(t)$ it makes sense to add several sinks at the same location $(0,i)$ so sinks may have a multiplicity.
\end{itemize}
Examples of dynamics of $L_<,L_\leq$ under the influence of sources/sinks is illustrated at the bottom of Fig.\ref{fig:4schemasHammersley}.

Here is the discrete-time analogous of \cite[Th.3.1.]{CatorGroeneboom}:
\begin{lem}\label{lem:Stationnaire<}
For every $\lambda,\alpha>0$ let $L^{(\alpha,p)}_{< }(t)$ be the Hammersley process defined as $L_{< }(t)$ with: 
\begin{itemize}
\item sources distributed according to a homogeneous PPP with intensity $\alpha$ on $[0,x]\times \{0\}$ ; 
\item sinks distributed according to i.i.d. $\mathrm{Bernoulli}(p)$ with
\begin{equation}\label{eq:alpha_p}
\frac{\lambda}{\lambda +\alpha}=p.
\end{equation}
\end{itemize}
If sources, sinks, and $\Pi^{(\lambda)}$ are independent then the process $\left(L^{(\alpha,p)}_{< }(t)\right)_{t\geq 0}$ is stationary.
\end{lem}

\begin{lem}\label{lem:Stationnaire_leq}
For every $\beta>\lambda>0$ let $L^{(\beta,\beta^\star)}_{\leq}(t)$ be the Hammersley process defined as $L_{\leq }(t)$ with: 
\begin{itemize}
\item sources distributed according to a homogeneous PPP with intensity $\beta$ on $[0,x]\times \{0\}$ ; 
\item sinks distributed according to i.i.d. $\mathrm{Geometric}_{\geq 0}(1-\beta^\star)$ with
\begin{equation}\label{eq:alpha}
\beta^\star\beta =\lambda.
\end{equation}
\end{itemize}
If sources, sinks, and $\Pi^{(\lambda)}$ are independent then the process $\left(L^{(\beta,\beta^\star)}_{\leq }(t)\right)_{t\geq 0}$ is stationary.
\end{lem}

\begin{proof}[Proof of Lemmas \ref{lem:Stationnaire<} and \ref{lem:Stationnaire_leq}]
Lemma \ref{lem:Stationnaire_leq} could be obtained from minor adjustments of \cite[Chap.3, Lemma 3.2]{TheseBoyer}. (Be aware that we have to switch $x\leftrightarrow t$ and sources $\leftrightarrow$ sinks in  \cite{TheseBoyer} in order to fit our setup.) 
For the sake of the reader we however propose the following alternative proof which explains where \eqref{eq:alpha} come from.

Consider for some fixed $t\geq 1$ the process $(H_y)_{0\leq y\leq x} $ given by the number of Hammersley lines passing through the point $(y,t)$. 

\begin{center}
\includegraphics[width=105mm]{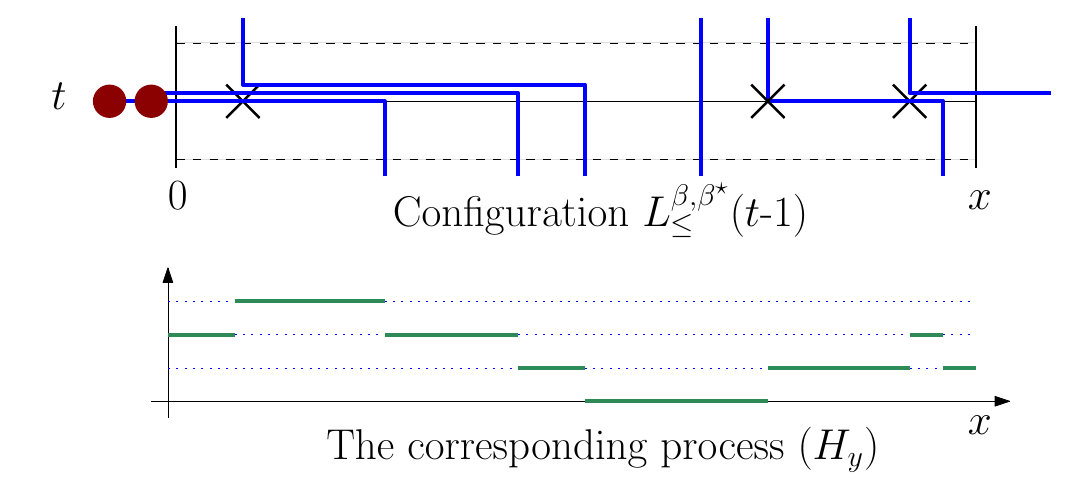}
\end{center}

The initial value $H_0$ is the number of sinks at $(0,t)$, which is distributed as a $\mathrm{Geometric}_{\geq 0}(1-\beta^\star)$. The process $(H_y)$ is a random walk (reflected at zero) with '$+1$ rate' equal to $\lambda$ and '$-1$ rate' equal to $\beta$. (Jumps of  $(H_y)$ are independent from sinks as sinks are independent from $\Pi^{(\lambda)}$.) The $\mathrm{Geometric}_{\geq 0}(1-\beta^\star)$ distribution is stationary for this random walk exactly when \eqref{eq:alpha} holds.
The set of points of $L^{(\beta,\beta^\star)}_{\leq }(t)$ is given  by the union of $\Pi_t^{(\lambda)}$ and the points of $L^{(\beta,\beta^\star)}_{\leq }(t)$ that do not correspond to a '$-1$' jump. Computations given in Appendix \ref{Sec:Appendix} show that this is distributed as a homogeneous PPP with intensity $\beta$.

Lemma \ref{lem:Stationnaire<} is proved exactly in the same way, calculations are even easier. In this case the corresponding process $(H_y)_{0\leq y\leq x} $ takes its values in $\{0,1\}$ and its stationary distribution is the Bernoulli distribution with mean $\lambda/(\alpha+\lambda)$, hence \eqref{eq:alpha_p}.
\end{proof}

\subsection{Processes $L_<(t)$ and $L_\leq(t)$: non-asymptotic bounds}

From Lemmas \ref{lem:Stationnaire<} and \ref{lem:Stationnaire_leq} it is straightforward to derive non-asymptotic upper bounds for $L_<(t),L_\leq(t)$. 

For $y\leq x$ let ${\sf So}^{(\alpha)}_x$ be the random set of sources with intensity $\alpha$ and for $s\leq t$ let  ${\sf Si}^{(p)}_t$ the random set   of sinks with intensity $p$. In particular,
$$
\mathrm{card}({\sf So}^{(\alpha)}_x)\stackrel{\text{(d)}}{=}\mathrm{Poisson}(\alpha x),\qquad \mathrm{card}({\sf Si}^{(p)}_t)\stackrel{\text{(d)}}{=}\mathrm{Binomial}(t,p).
$$
 It is convenient to use the notation $\mathcal{L}_{=<}(\mathcal{P})$
which is, as before, the length of the longest increasing path taking points in $\mathcal{P}$ but when the path is also allowed to go through several sources (which have however the same $y$-coordinate) or several sinks (which have the same $x$-coordinate). Formally,
$$
\mathcal{L}_{=<}(\mathcal{P})=\max\set{L; \hbox{ there exists }P_1  =\prec P_2 =\prec \dots =\prec P_L, \text{ where each }P_i \in \mathcal{P} },\\
$$
where
$$
(x,y)=\prec (x',y') \text{ if }
\begin{cases}
 &x<x'\text{ and }y<y', \\
\text{ or } &x=x'=0 \text{ and }y<y', \\
\text{ or } &x<x' \text{ and }y=y'=0.
\end{cases}
$$
Proposition \ref{prop:NombreLignes} generalizes easily to the settings of sinks and sources.
\begin{claim*}
\begin{equation}\label{prop:NombreLignes_bis}
\mathcal{L}_{=<}\left( \Pi^{(\lambda)}_{x,t}\cup {\sf So}^{(\alpha)}_x\cup {\sf Si}^{(p)}_t\right) =L^{(\alpha,p)}_{< }(t)+\mathrm{card}({\sf Si}^{(p)}_t).
\end{equation}
\end{claim*}
\begin{proof}[Proof of the Claim.] By the same reasoning as in the proof of Proposition \ref{prop:NombreLignes} the LHS is exactly the number of broken lines in the box $[0,x]\times [0,t]$. Each such line escapes the box either through the left (it thus corresponds to a sink) or through the top (and is thus counted by $L^{(\alpha,p)}_{< }(t)$).
\end{proof}

\begin{lem}[Domination for $\mathcal{L}_<$]\label{lem:Majo_L_sources}
For every $\alpha,p \in(0,1)$ such that \eqref{eq:alpha_p} holds,
there is a stochastic domination of the form: 
\begin{equation}\label{eq:Majo_L_sources}
\mathcal{L}_<\left( \Pi^{(\lambda)}_{x,t}\right) \preccurlyeq  \Poi(x\alpha)+ \Bin(t,p).
\end{equation}
(The $ \Poi$ and $\Bin$ random variables involved in \eqref{eq:Majo_L_sources} are not independent.)

\end{lem}
\begin{proof}
Adding sources and sinks may not decrease longest increasing paths. Thus,
\begin{align*}
\mathcal{L}_<\left( \Pi^{(\lambda)}_{x,t}\right) &\preccurlyeq \mathcal{L}_{=<}\left( \Pi^{(\lambda)}_{x,t}\cup {\sf So}^{(\alpha)}_x\cup {\sf Si}^{(p)}_t\right) \\
&= L^{(\alpha,p)}_{< }(t)+\mathrm{card}({\sf Si}^{(p)})\text{ (using \eqref{prop:NombreLignes_bis})}\\
&\stackrel{\text{(d)}}{=} L^{(\alpha,p)}_{< }(0)+\mathrm{card}({\sf Si}^{(p)}) \text{ (using stationarity: Lemma \ref{lem:Stationnaire<})}\\
&\stackrel{\text{(d)}}{=}  \Poi(x\alpha)+ \Bin(t,p).\qedhere
\end{align*}
\end{proof}
Taking expectations in \eqref{eq:Majo_L_sources} we obtain
$$
\mathbb{E}\left[\mathcal{L}_<\left(\Pi^{(\lambda)}_{x,t}\right)\right] \leq  x\alpha +tp.
$$
The LHS in the above equation  does not depend on $\alpha,p$ so the idea is to apply \eqref{eq:Majo_L_sources} with the minimizing choice
$$
\bar{\alpha},\bar{p}:=\mathrm{argmin}_{\alpha,p\text{ satisfying \eqref{eq:alpha_p}}} \left\{ x\alpha +tp\right\},
$$
\emph{i.e.}
\begin{equation}\label{eq:alpha_p_bar}
\bar{\alpha}=\sqrt{\frac{t\lambda}{x}}-\lambda,\qquad \bar{p}=\sqrt{\frac{x\lambda}{t}},\qquad  x\bar{\alpha} + t\bar{p}=2\sqrt{xt\lambda}-x\lambda.
\end{equation}
We have proved
$$
\mathbb{E}\left[\mathcal{L}_<\left(\Pi^{(\lambda)}_{x,t}\right)\right] \leq  2\sqrt{xt\lambda}-x\lambda.
$$
(Compare with \eqref{eq:Theta<}.) We have a similar statement for non-decreasing subsequences:
\begin{lem}[Domination for $\mathcal{L}_\leq$]\label{lem:Majo_L_leq_sources}
For every $\beta,\beta^\star \in(0,1)$  such that \eqref{eq:alpha} holds,
there is a stochastic domination of the form: 
\begin{equation}\label{eq:Majo_L_leq_sources}
\mathcal{L}_\leq\left( \Pi^{(\lambda)}_{x,t}\right) \preccurlyeq  \Poi(x\beta)+ \mathcal{G}_1^{(\beta^\star)}+ \dots +  \mathcal{G}_t^{(\beta^\star)},
\end{equation}
where $\mathcal{G}_i^{(\beta^\star)}$'s are i.i.d. $\mathrm{Geometric}_{\geq 0}(1-\beta^\star)$. 
\end{lem}
We put
\begin{equation}\label{eq:alpha_bar}
\bar{\beta},\bar{\beta}^\star:=\mathrm{argmin}_{\beta,\beta^\star\text{ satisfying \eqref{eq:alpha}}} \left\{ x\beta + t\left(\frac{\beta^\star}{1-\beta^\star}\right)\right\},
\end{equation}
\emph{i.e.}
\begin{equation}\label{eq:alpha_alpha_bar}
\bar{\beta}=\sqrt{\frac{t\lambda}{x}}+\lambda,\qquad \bar{\beta}^\star=\frac{1}{1+\sqrt{t/x\lambda}},
\qquad  x\bar{\beta} + t\left(\frac{\bar{\beta}^\star}{1-\bar{\beta}^\star}\right)=2\sqrt{xt\lambda}+x\lambda.
\end{equation}
(In particular $\bar{\beta}>\lambda$, as required in Lemma \ref{lem:Stationnaire_leq}.)
Eq.\eqref{eq:Majo_L_leq_sources} yields
\begin{equation}\label{eq:esperance_upper_leq}
\mathbb{E}\left[\mathcal{L}_\leq\left( \Pi^{(\lambda)}_{x,t}\right)\right] \leq  2\sqrt{xt\lambda}+x\lambda.
\end{equation}
(Compare with \eqref{eq:Theta_leq}.) 


\begin{theo}[Concentration for $\mathcal{L}_<$, $\mathcal{L}_\leq$]\label{Th:BorneInf_L}

There exist strictly positive functions $g,h$ such that for all $\varepsilon>0$ and for every $x,t\geq 1$, $\lambda >0$ such that $t\geq x\lambda$
\begin{align}
\P(\mathcal{L}_{<}(\Pi^{(\lambda)}_{x,t}) &>(1+\eps)(2\sqrt{xt\lambda}-x\lambda) )\le \exp(-g(\varepsilon)(\sqrt{xt\lambda}-x\lambda)) \label{eq:upper},\\
\P(\mathcal{L}_{<}(\Pi^{(\lambda)}_{x,t}) &<(1-\eps)(2\sqrt{xt\lambda}-x\lambda) )\le \exp(-h(\varepsilon)(\sqrt{xt\lambda}-x\lambda)).\label{eq:lower}
\end{align}
Similarly:
\begin{align}
\P(\mathcal{L}_{\leq}(\Pi^{(\lambda)}_{x,t}) &>(1+\eps)(2\sqrt{xt\lambda}+x\lambda) )\le \exp(-g(\varepsilon)\sqrt{xt\lambda}) \label{eq:upper_leq},\\
\P(\mathcal{L}_{\leq}(\Pi^{(\lambda)}_{x,t}) &<(1-\eps)(2\sqrt{xt\lambda}+x\lambda) )\le \exp(-h(\varepsilon)\sqrt{xt\lambda}).\label{eq:lower_leq}
\end{align}
\end{theo}

For the proof of Theorem \ref{Th:BorneInf_L} we will focus on the case of $\mathcal{L}_<$, \emph{i.e.} eq.\eqref{eq:upper}, \eqref{eq:lower}. When necessary we will give the slight modification needed to prove eq.\eqref{eq:upper_leq} and \eqref{eq:lower_leq}.
The beginning of the proof mimics Lemmas 4.1 and 4.2 in \cite{NousAlea}.

We first prove similar  bounds for the stationary processes with minimizing sources and sinks.

\begin{lem}[Concentration for $\mathcal{L}_<$ with sources and sinks]\label{lem:lower_upper_star}
Let $\bar{\alpha},\bar{p}$ be defined by \eqref{eq:alpha_p_bar}.
There exists  a strictly positive function $g_1$ such that for all $\varepsilon>0$ and for every $x,t\geq 1$, $\lambda >0$ such that $t\geq x\lambda$\begin{align}
\P(\mathcal{L}_{=<}(\Pi^{(\lambda)}_{x,t}\cup {\sf So}^{(\bar{\alpha})}_x\cup {\sf Si}^{(\bar{p})}_t) &>(1+\eps)(2\sqrt{xt\lambda}-x\lambda) )\le 2\exp(-g_1(\varepsilon)(\sqrt{xt\lambda}-x\lambda)) \label{eq:upper_star}\\
\P(\mathcal{L}_{=<}(\Pi^{(\lambda)}_{x,t}\cup {\sf So}^{(\bar{\alpha})}_x\cup {\sf Si}^{(\bar{p})}_t) &<(1-\eps)(2\sqrt{xt\lambda}-x\lambda) )\le 2\exp(-g_1(\varepsilon)(\sqrt{xt\lambda}-x\lambda)) \label{eq:lower_star}.
\end{align} 
\end{lem}
\begin{proof}[Proof of Lemma \ref{lem:lower_upper_star}]

By stationarity (Lemma \ref{lem:Stationnaire<}) we have
 $$
\mathcal{L}_{=<}(\Pi^{(\lambda)}_{x,t}\cup {\sf So}^{(\bar{\alpha})}_x\cup {\sf Si}^{(\bar{p})}_t) \overset{(d)}{=}\Poi(x\bar{\alpha})+\Bin(t,\bar{p}).
 $$
 Then
\begin{align*}
 \P(\mathcal{L}_{=<}(\Pi^{(\lambda)}_{x,t}\cup {\sf So}^{(\bar{\alpha})}_x\cup {\sf Si}^{(\bar{p})}_t>(1+\eps)(2\sqrt{xt\lambda}-x\lambda))
\le  &\ \P\left(\Poi(x\bar{\alpha})>(1+\frac{\varepsilon}{2})(\sqrt{xt\lambda}-x\lambda)\right)\\
+&\ \P\left(\Bin(t,\bar{p})>(1+\frac{\varepsilon}{2})\sqrt{xt\lambda}\right).
\end{align*}

Recall that $x\bar{\alpha}=\sqrt{xt\lambda}-x\lambda$, $t\bar{p}=\sqrt{xt\lambda}$. Using the tail inequality for the Poisson distribution (Lemma \ref{lem:Chernov} (i)):
\begin{align*}
\P\left(\Poi(x\bar{\alpha})>(1+\frac{\varepsilon}{2})(\sqrt{xt\lambda}-x\lambda)\right)&\le  
\exp\left(-(\sqrt{xt\lambda}-x\lambda)\eps^2/4\right).
\end{align*}
Using  the tail inequality for the binomial (Lemma \ref{lem:ChernovB}) we get 
\begin{equation}\label{eq:Binom}
\P\left(\Bin(t,\bar{p})>(1+\frac{\varepsilon}{2})\sqrt{xt\lambda}\right) \le  \exp(-\tfrac{1}{12}\varepsilon^2\sqrt{xt\lambda})\leq  \exp(-\tfrac{1}{12}\varepsilon^2(\sqrt{xt\lambda}-x\lambda))
\end{equation}
The proof of  \eqref{eq:lower_star} is identical.
This shows Lemma \ref{lem:lower_upper_star} with $g_1(\eps)=\eps^2/12$.
 \end{proof}
 
For longest non-decreasing subsequences we have a statement similar to Lemma \ref{lem:lower_upper_star}. The only modification in the proof is that in order to estimate the number of sinks one has to replace  Lemma \ref{lem:ChernovB} (tail inequality for the Binomial) by Lemma \ref{lem:ChernovG} (tail inequality for a sum of geometric random variables\footnote{Note that it is only stated for $0<\eps<1$ but this is enough for our purpose since the left-hand side of \eqref{eq:Binom} is non-increasing in $\eps$.}). During the proof we need to bound $\sqrt{xt\lambda}+x\lambda$ by $\sqrt{xt\lambda}$, this explains the form of the right-hand side in eq.\eqref{eq:upper_leq} and \eqref{eq:lower_leq}.
 
 \begin{proof}[Proof of Theorem \ref{Th:BorneInf_L}]
Adding sources/sinks may not decrease $\mathcal{L}_{<}$ so
$$
\mathcal{L}_{=<}(\Pi^{(\lambda)}_{x,t}\cup {\sf So}^{(\bar{\alpha})}_x\cup {\sf Si}^{(\bar{p})}_t)
\succcurlyeq \mathcal{L}_{<}(\Pi^{(\lambda)}_{x,t}),
$$
thus the upper bound \eqref{eq:upper} is a direct consequence of Lemma \ref{lem:lower_upper_star}.

%

\begin{figure}
\begin{center}
\includegraphics[width=160mm]{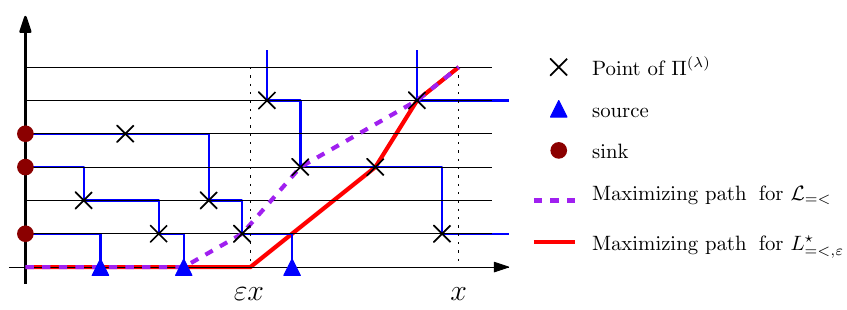} 
\end{center}
\caption{A sample of $\Pi^{(\lambda)}_{x,t}$, sources, sinks, and the corresponding trajectories of particles (in blue).  Here $\mathcal{L}_{=<}(\Pi^{(\lambda)}_{x,t}\cup {\sf So}^{(\alpha)}_x\cup {\sf Si}^{(p)}_t)=5$ (pink path) and $L^{(\alpha,p)}_{< }(t)=2$ (two remaining particles at the top of the box).}
\label{fig:SourcesForcees}
\end{figure}

Let us now prove the lower bound. We consider the length of a maximizing path among those using sources from $0$ to $\eps x$ and then only increasing points of $\Pi^{(\lambda)}_{x,t}\cap \left([\eps x,x]\times [0,t]\right)$ (see Fig.\ref{fig:SourcesForcees}). Formally we set
\begin{align}
L_{=<,\eps}^\star&:=
\mathrm{card}\left({\sf So}^{(\bar{\alpha})}_{\eps x}\right)+
\mathcal{L}_{<}\left(  (\Pi^{(\lambda)}_{x,t}\cap \left([\eps x,x]\times [0,t]\right)\right)\notag\\
&\stackrel{\text{(d)}}{=}
\Poi(\eps x\bar{\alpha}) + \mathcal{L}_{<}\left(  (\Pi^{(\lambda)}_{x,t}\cap \left([\eps x,x]\times [0,t]\right)\right).
\label{eq:SourcesForcees}
\end{align}
The idea is that for any fixed $\eps$ the paths contributing to $L_{=<,\eps}^\star$ will typically not contribute to $\mathcal{L}_{=<}\left( \Pi^{(\lambda)}_{x,t}\cup {\sf So}^{(\bar{\alpha})}_x\cup {\sf Si}^{(\bar{p})}_t\right) =L^{(\bar{\alpha},p)}_{< }(t)+\mathrm{card}({\sf Si}^{(\bar{p})}_t)$. Indeed eq.\eqref{eq:SourcesForcees} suggests that for large $x,t$
\begin{align*}
L_{=<,\eps}^\star&\approx \mathbb{E}[\Poi(\eps x\bar{\alpha})] + \mathbb{E}\left[\mathcal{L}_{<}\left(  \Pi^{(\lambda)}_{x,t}\cap \left([\eps x,x]\times [0,t]\right)\right)\right]\\
&\approx x\eps\bar{\alpha} + 2\sqrt{x(1-\eps)\lambda t}-x(1-\eps)\lambda\\
& = 2\sqrt{x\lambda t} -x\lambda -\sqrt{xt\lambda}\delta(\eps),
\end{align*}
where $\delta(\eps)=2-\eps-2\sqrt{1-\eps}\leq \eps^2$ is positive and increasing. In order to make the above approximation rigorous we first write
\begin{equation}\label{eq:useful}
 2\sqrt{x\lambda t} -x\lambda -\tfrac{1}{2}\sqrt{xt\lambda}\delta(\eps)
 = x\eps\bar{\alpha} + \tfrac{1}{4}\sqrt{xt\lambda}\delta(\eps)
+  2\sqrt{x(1-\eps)\lambda t} -x(1-\eps)\lambda +\tfrac{1}{4}\sqrt{xt\lambda}\delta(\eps).
\end{equation}
Combining \eqref{eq:SourcesForcees} and \eqref{eq:useful} gives
$$
\P\left(L_{=<,\eps}^\star\ge  2\sqrt{x\lambda t} -x\lambda -\tfrac{1}{2}\sqrt{xt\lambda}\delta(\eps)\right)\notag\\
\leq \P_1+ \P_2
$$
where
\begin{align*}
\P_1&= \P\left(\Poi(x\eps\bar{\alpha})\ge x\eps\bar{\alpha} + \tfrac{1}{4}\sqrt{xt\lambda}\delta(\eps)\right),\\
\P_2&=\P\left( \mathcal{L}_{<} (\Pi^{(\lambda)}_{x,t}\cap \left([\eps x,x]\times [0,t]\right) \geq 2\sqrt{x(1-\eps)\lambda t} -x(1-\eps)\lambda +\tfrac{1}{4}\sqrt{xt\lambda}\delta(\eps)\right).
\end{align*}
Using the tail inequality for the Poisson distribution (Lemma \ref{lem:Chernov} (i)) we have that
$$
\P_1\leq \exp\left(-\frac{xt\lambda \delta(\eps)^2 }{16\times 4 \eps^2(\sqrt{xt\lambda}-x\lambda)}\right)\leq  \exp\left(-\sqrt{xt\lambda} \delta(\eps)^2 /64 \eps^2\right).
$$
Besides
\begin{align}
\P_2
 &\leq \P\left( \mathcal{L}_{<} (\Pi^{(\lambda)}_{x,t}\cap \left([\eps x,x]\times [0,t]\right) \geq \left(2\sqrt{x(1-\eps)\lambda t} -x(1-\eps)\lambda\right)\times (1+\tfrac{1}{8}\delta(\eps)) \right)\notag\\
&\leq \exp\left(-g(\delta(\eps)/8) (\sqrt{x(1-\eps)t\lambda}-x(1-\eps)\lambda) \right)  \text{(using the upper bound \eqref{eq:upper})}.\notag
\end{align}
Finally we can find  some positive  $h$ such that 
\begin{equation}\label{eq:Majo_Leps}
\P\left(L_{=<,\eps}^\star\ge  2\sqrt{x\lambda t} -x\lambda -\tfrac{1}{2}\sqrt{xt\lambda}\delta(\eps)\right)\leq \exp\left(-h(\eps) (\sqrt{xt\lambda}-x\lambda)\right).
\end{equation}
One proves exactly in the same way a  similar bound  for the length of a maximizing path among those using sinks in $\{0\}\times [0,\eps t]$ and then only increasing points of $\Pi^{(\lambda)}_{x,t}\cap \left([0,x]\times [\eps t,t]\right)$.


Choose now one of the maximizing paths $\mathcal{P}$ for $\mathcal{L}_{=<}\left( \Pi^{(\lambda)}_{x,t}\cup {\sf So}^{(\bar{\alpha})}_x\cup {\sf Si}^{(\bar{p})}_t\right)$ (if there are many of them, choose one arbitrarily in a deterministic way: the lowest, say). 
Denote by ${\sf sources}(\mathcal{P})$ and ${\sf sinks}(\mathcal{P})$ the number of sources and sinks in the path $\mathcal{P}$:
$$
{\sf sources}(\mathcal{P})=\mathrm{card}\set{0\leq y\leq x \text{ such that } (y,0)\in \mathcal{P} }.
$$
In Fig.\ref{fig:SourcesForcees} the path $\mathcal{P}$ is sketched in pink, in that example ${\sf sources}(\mathcal{P})=2$, ${\sf sinks}(\mathcal{P})=0$.

\begin{lem}\label{lem:PeuDeSources}
Let $t\geq x\lambda$. There exists a positive function $\psi$ such that for all real $\eta >0$
\begin{equation*}\label{eq:PeuDeSources}
\P\left({\sf sources}(\mathcal{P})+{\sf sinks}(\mathcal{P})\geq  \eta \sqrt{x\lambda t}\right) \leq 2\exp(-\psi(\eta)(\sqrt{x\lambda t}-x\lambda)).
\end{equation*}
\end{lem}

\begin{proof}[Proof of Lemma \ref{lem:PeuDeSources}]
(As the left-hand side is non-increasing in $\eta$ it is enough to prove the lemma for $\eta<1$.)\\

If the event $\left\{{\sf sources}(\mathcal{P})\geq \eta \sqrt{x\lambda t}\right\}$ holds then there exists a (random) $\eps$ such that the two following events occur:
\begin{itemize}
\item ${\sf So}_{\eps x}\geq \eta \sqrt{x\lambda t}$ ;
\item $L_{=<,\eps}^\star =\mathcal{L}_{=<}\left( \Pi^{(\lambda)}_{x,t}\cup {\sf So}^{(\bar{\alpha})}_x\cup {\sf Si}^{(\bar{p})}_t\right) =L^{(\bar{\alpha},\bar{p})}_{< }(t)+\mathrm{card}({\sf Si}^{(\bar{p})}_t)$.
\end{itemize}
This implies that this random $\eps$ is larger than  $\eta/2>0$ unless the number of sources in  $[0,x\eta/2]$ is improbably high:
$$
\begin{array}{r c l l}
\mathbb{P}({\sf sources}(\mathcal{P})\geq \eta \sqrt{x\lambda t})
&\leq &\ \mathbb{P}({\sf So}_{\eta x/2}\geq  \eta \sqrt{x\lambda t}) +  \mathbb{P}({\sf sources}(\mathcal{P})\geq \eta \sqrt{x\lambda t};\ {\sf So}_{\eta x/2}<  \eta \sqrt{x\lambda t}) & \\
&\leq &\ \mathbb{P}({\sf So}_{\eta x/2}\geq  \eta \sqrt{x\lambda t}) & =:\P_3\\ 
&+ &\ \mathbb{P}(L^{(\bar{\alpha},\bar{p})}_{<}(t) \leq \sqrt{x\lambda t} -x\lambda - \tfrac{1}{4}\delta(\eta/3)\sqrt{x\lambda t}) &  =:\P_4\\
&+ &\ \mathbb{P}(\mathrm{card}({\sf Si}^{(\bar{p})}_t)\leq  \sqrt{x\lambda t} - \tfrac{1}{4}\delta(\eta/3)\sqrt{x\lambda t})&  =:\P_5\\
&+ &\ \mathbb{P}(L_{=<,\eps}^\star \geq 2\sqrt{x\lambda t} -x\lambda -\tfrac{1}{2}\delta(\eta/3)\sqrt{x\lambda t}\text{ for some }\eta/2\leq \eps \leq 1 )&  =:\P_6.
\end{array}
$$
From previous calculations, the three first terms in the above display are less than $\exp(-\phi(\eta)(\sqrt{x\lambda t}-x\lambda))$ for some positive function $\phi$. To see why:
\begin{itemize}
\item We bound $\P_3$ with Lemma \ref{lem:Chernov} (i) again. Here we need to justify that the condition "$A\leq 3\lambda$" in  
 Lemma \ref{lem:Chernov} (i) is satisfied. The quantity ${\sf So}_{\eta x/2}$ is a Poisson random variable with mean $\bar{\alpha}\frac{\eta}{2}x$ so we have to check that $\eta \sqrt{x\lambda t}-\bar{\alpha}\frac{\eta}{2}x\leq 3\bar{\alpha}\frac{\eta}{2}x$. Recalling $\bar{\alpha}=\sqrt{t\lambda/x}-\lambda$ and a bit of algebra shows that this is equivalent to $t\geq x\lambda$.
\item The term $\P_4$ is bounded thanks to Lemma \ref{lem:lower_upper_star} (recall also \eqref{prop:NombreLignes_bis});
\item We bound $\P_5$ with Lemma \ref{lem:ChernovB} (recall that ${\sf Si}^{(\bar{p})}_t$ is a Binomial).
\end{itemize}
To conclude the proof it remains to bound $\P_6$. Let $K$ be an integer larger than $144/\eta^3$, by definition of $L_{=<,\eps}^\star$ we have for every $1\leq k\leq \lceil xK \rceil$ and every $\eps \in[\tfrac{k}{K},\tfrac{k+1}{K})$ 
$$
L_{=<,\eps}^\star \leq L_{=<,k/K}^\star + \mathrm{card}({\sf So}^{(\bar{\alpha})}_x \cap [\tfrac{k}{K},\tfrac{k+1}{K}]).
$$
Thus
\begin{align}
\mathbb{P}\bigg(\bigcup_{\eta/2\leq \eps \leq 1} & \left\{L_{=<,\eps}^\star > 2\sqrt{x\lambda t} -x\lambda -\tfrac{1}{2}\delta(\eta/3)\sqrt{x\lambda t}\right\} \bigg)\notag\\
\leq &\sum_{k\geq \lfloor \eta K/2\rfloor } \mathbb{P}\left(L_{=<,k/K}^\star > 2\sqrt{x\lambda t} -x\lambda -\delta(\eta/3)\sqrt{x\lambda t} \right)\notag\\
+ &\sum_{k\geq \lfloor \eta K/2\rfloor } \mathbb{P}\left(\mathrm{card}({\sf So}^{(\bar{\alpha})}_x \cap [\tfrac{k}{K},\tfrac{k+1}{K}]) > \tfrac{1}{2}\delta(\eta/3)\sqrt{x\lambda t}\right)\notag\\
\leq &\sum_{k\geq \lfloor \eta K/2\rfloor } \mathbb{P}\left(L_{=<,k/K}^\star > 2\sqrt{x\lambda t} -x\lambda -\delta(k/K)\sqrt{x\lambda t} \right) \quad\text{ (since $K>\frac{144}{\eta^3}>\frac{6}{\eta}$ and $\delta$ is increasing)}\notag\\
+ &\sum_{k\geq \lfloor \eta K/2\rfloor } \mathbb{P}\left(\mathrm{card}({\sf So}^{(\bar{\alpha})}_x \cap [\tfrac{k}{K},\tfrac{k+1}{K}]) > \tfrac{1}{2}\delta(\eta/3)\sqrt{x\lambda t}\right)\notag\\
\leq &\sum_{k\geq \lfloor  \eta K/2\rfloor } \exp(-h(k/K) (\sqrt{xt\lambda}-x\lambda))\qquad\text{ (using \eqref{eq:Majo_Leps})}\notag\\
+ &K \times \mathbb{P}\left(\mathrm{Poisson}(\bar{\alpha}/K)>\tfrac{1}{2}\delta(\eta/3)\sqrt{x\lambda t}\right)\notag\\
\leq &\ K \exp(-h(\eta/3) (\sqrt{xt\lambda}-x\lambda))  \label{eq:LastLineBigDisplay}\\
+  & K \times \mathbb{P}\left(\mathrm{Poisson}(\bar{\alpha}/K)> \tfrac{1}{2}\delta(\eta/3)\sqrt{x\lambda t}\right).\notag
\end{align}
We finally bound the last display. First recall from our notation that
$$
\bar{\alpha}<\sqrt{t\lambda/x},\qquad x\geq 1,\qquad \delta(\eps)=2-\eps-2\sqrt{1-\eps}\geq \eps^2/4.
$$
Then:
\begin{align}
 \mathbb{P}\left(\mathrm{Poisson}(\bar{\alpha}/K)> \tfrac{1}{2}\delta(\eta/3)\sqrt{x\lambda t}\right)
 &=\mathbb{P}\left(\mathrm{Poisson}(\bar{\alpha}/K)> \bar{\alpha}/K-\bar{\alpha}/K+\tfrac{1}{2}\delta(\eta/3)\sqrt{x\lambda t}\right)\notag\\ 
&\leq \mathbb{P}\left(\mathrm{Poisson}(\bar{\alpha}/K)> \bar{\alpha}/K+\sqrt{x\lambda t}\left( -\frac{1}{xK}+\frac{1}{2}\delta(\eta/3)\right) \right)\notag\\
&\leq \mathbb{P}\left(\mathrm{Poisson}(\bar{\alpha}/K)> \bar{\alpha}/K+\sqrt{x\lambda t}\left( -\frac{\eta^3}{144}+\frac{\eta^2}{72}\right) \right).\label{eq:Fin144}
\end{align}
We can find a positive function $\varphi$ such that \eqref{eq:LastLineBigDisplay} and \eqref{eq:Fin144} are both less than $\frac{144}{\eta} e^{-\varphi(\eta) (\sqrt{xt\lambda}-x\lambda)}$.
We then choose a positive function $\psi$ such that 
$$
 \min\left\{1,\frac{288}{\eta} e^{-\varphi(\eta) (\sqrt{xt\lambda}-x\lambda)}+ 3e^{-\phi(\eta)(\sqrt{x\lambda t}-x\lambda)}\right\}\leq 2e^{-\psi(\eta) (\sqrt{xt\lambda}-x\lambda)}
$$
and thus $\mathbb{P}({\sf sources}(\mathcal{P})\geq \eta \sqrt{x\lambda t})\leq \exp(-\psi(\eta) (\sqrt{xt\lambda}-x\lambda))$.
With minor modifications one proves the same bound for sinks (possibly by changing $\psi$): 
$ \mathbb{P}({\sf sinks}(\mathcal{P})\geq \eta \sqrt{x\lambda t})\leq \exp(-\psi(\eta) (\sqrt{xt\lambda}-x\lambda))$ and Lemma \ref{lem:PeuDeSources} is proved.
\end{proof}

We can conclude the proof of  the lower bound in Theorem \ref{Th:BorneInf_L}. Let us write
 $$
 L_<(t)\geq \mathcal{L}_{=<}(\Pi^{(\lambda)}_{x,t}\cup {\sf So}^{(\bar{\alpha})}_x\cup {\sf Si}^{(\bar{p})}_t)-{\sf sources}(\mathcal{P})-{\sf sinks}(\mathcal{P}),
 $$
we bound the right-hand side using Lemmas \ref{lem:lower_upper_star} and \ref{lem:PeuDeSources}.
\end{proof}

\section{Proof of Theorem \ref{Th:Strict} when $k_n\to +\infty$: de-Poissonization}\label{Sec:Dep}

In order to conclude the proof of Theorem \ref{Th:Strict} it remains to de-Poissonize Theorem \ref{Th:BorneInf_L}.
We need a few notation. For any integers $i_1,\dots,i_n$ let $\mathcal{S}_{i_1,\dots,i_n}$ be the random set of points given by $i_\ell$ uniform points on each horizontal line:
$$
\mathcal{S}_{i_1,\dots,i_n}=
\cup_{\ell=1}^n \cup_{r=1}^{i_\ell} \set{U_{\ell,r}}\times\set {\ell},
$$
where $(U_{\ell,r})_{\ell,r}$ is an array of i.i.d. uniform random variables in $[0,1]$. Set also
$
e_{i_1,\dots,i_n} =\mathbb{E}[\mathcal{L}_{<}(\mathcal{S}_{i_1,\dots,i_n})].
$
By uniformity of $U$'s we have the identity
$
\mathbb{E}[\mathcal{L}_{<}(S_{k;n})]=e_{k,\dots,k}
$
and therefore our problem reduces to estimating $e_{k,\dots,k}$. On the other hand if $X_1,\dots,X_n$ are i.i.d. Poisson random variables with mean $k$ then
\begin{align}
\mathbb{E}[e_{X_1,\dots,X_n}]=
\mathbb{E}\left[\mathcal{L}_{<}(\Pi_{nk_n,n}^{(1/n)})\right]
= 2\sqrt{nk_n}-k_n + o(\sqrt{nk_n}).
\label{eq:transformee_poisson}
\end{align}
The last equality is obtained by combining Theorem \ref{Th:BorneInf_L} for
$$
x=nk_n,\qquad t=n,\qquad \lambda_n = \frac{1}{n}
$$
with the trivial bound $\mathcal{L}_{<}(\Pi_{nk_n,n}^{(1/n)})\leq n$. In order to exploit \eqref{eq:transformee_poisson} we need the following smoothness estimate.
\begin{lem}
For every $i_1,\dots,i_n$ and $j_1,\dots,j_n$
$$
\left| e_{i_1,\dots,i_n} - e_{j_1,\dots,j_n} \right| \leq 6\sqrt{\sum_{\ell=1}^n |i_\ell-j_\ell |}.
$$
\end{lem}
\begin{proof}
Let $\mathcal{S}=\mathcal{S}_{i_1,\dots,i_n}$ be as above. If we replace in $\mathcal{S}$ the $y$-coordinate of each point of the form $(x,\ell)$ by a new $y$-coordinate uniform in the interval  $(\ell,\ell+1)$ (independent from anything else) then this defines a uniform permutation $\sigma_{i_1+\dots +i_n}$ of size $i_1+\dots +i_n$. The longest increasing subsequence in $\mathcal{S}$ is mapped onto an increasing subsequence in $\sigma_{i_1+\dots +i_n}$ and thus this construction shows  the stochastic domination
$
\mathcal{L}_{<}(\mathcal{S}_{i_1,\dots,i_n}) \preccurlyeq \mathcal{L}_{<}(\sigma_{i_1+\dots +i_n}).
$
Thus for every $i_1,\dots,i_n$,
\begin{equation}\label{eq:CrudeBound}
e_{i_1,\dots,i_n}\leq \mathbb{E}[\mathcal{L}_{<}(\sigma_{i_1+\dots +i_n})] \leq 6\sqrt{i_1+\dots +i_n}.
\end{equation}
(The second inequality follows for example from \cite[Lemma 1.4.1]{Steele}.)
Besides, consider for two $n$-tuples $i_1,\dots,i_n$ and $j_1,\dots,j_n$ two independent sets of points $\mathcal{S}_{i_1,\dots,i_n}$, $ \widetilde{\mathcal{S}}_{j_1,\dots,j_n}$ then 
$$
\mathcal{L}_{<}(\mathcal{S}_{i_1,\dots,i_n})\leq \mathcal{L}_{<}(\mathcal{S}_{i_1,\dots,i_n}\cup \widetilde{\mathcal{S}}_{j_1,\dots,j_n})
\leq \mathcal{L}_{<}(\mathcal{S}_{i_1,\dots,i_n}) +\mathcal{L}_{<}( \widetilde{\mathcal{S}}_{j_1,\dots,j_n}).
$$
This proves that
$$
e_{i_1,\dots,i_n}\leq e_{i_1+j_1,\dots,i_n+j_n} \leq  e_{i_1,\dots,i_n}+e_{j_1,\dots,j_n}.
$$
(In particular $(i_1,\dots,i_n)\mapsto e_{i_1,\dots,i_n}$ is non-decreasing with respect to any of its coordinate.) Therefore
\begin{align*}
e_{i_1,\dots,i_n}
&\leq e_{(i_1-j_1)^+,\dots,(i_n-j_n)^+} +  e_{j_1-(i_1-j_1)^-,\dots,j_n-(i_n-j_n)^-}\\
&\leq e_{|i_1-j_1|,\dots,|i_n-j_n|}  + e_{j_1,\dots,j_n}.
\end{align*}

By switching the role of $i$'s and $j$'s:
$$
|e_{i_1,\dots,i_n}-e_{j_1,\dots,j_n}|  \leq e_{|i_1-j_1|,\dots,|i_n-j_n|}
\leq 6\sqrt{\sum_{\ell=1}^n |i_\ell-j_\ell |},
$$
using  \eqref{eq:CrudeBound}.
\end{proof}

\begin{proof}[Proof of Theorem \ref{Th:Strict} for any sequence $(k_n)\to +\infty$]
Using smoothness we write
\begin{align}
|e_{k,\dots,k} - \mathbb{E}[e_{X_1,\dots,X_n}]|
\leq \mathbb{E}\left[|e_{k,\dots,k} - e_{X_1,\dots,X_n}|\right]
\leq  6\times \mathbb{E}\left[ \left(\sum_{\ell=1}^n | X_\ell-k |\right)^{1/2}\right].\label{eq:depoisson}
\end{align}
Using twice the Cauchy-Schwarz inequality:
\begin{align*}
\mathbb{E}\left[ \left(\sum_{\ell=1}^n | X_\ell-k |\right)^{1/2}\right]
&\leq \sqrt{\mathbb{E}\left[ \sum_{\ell=1}^n | X_\ell-k |\right]}\\
&\leq \sqrt{n \mathbb{E}\left[ | X_1-k |\right]}\\
&\leq \sqrt{n \mathbb{E}\left[ | X_1-k |^{2}\right]^{1/2}}
= \sqrt{n \sqrt{\mathrm{Var}(X_1)}}
= \sqrt{n \sqrt{k}}.
\end{align*}
If  $k=k_n\to \infty$ then the last display is a $o(\sqrt{nk_n})$ and  eq.\eqref{eq:depoisson} and \eqref{eq:transformee_poisson} show that
$$
e_{k,\dots,k}=\mathbb{E}[\mathcal{L}_{<}(S_{k;n})]=2\sqrt{nk_n}-k_n + o(\sqrt{nk_n}).
$$\qedhere
\end{proof}

\section{Proof of Theorem \ref{Th:Large}}\label{Sec:Dep2}

\subsection{Proof for large $(k_n)$}\label{Sec:DepoissonizationLarge}

We now prove Theorem \ref{Th:Large} for a \emph{large} sequence  $(k_n)$.
We say that $(k_n)$ is \emph{large} if 
\begin{equation}\label{eq:condition_large}
n^2k_n\exp(-(k_n)^{\alpha})=\mathrm{o}(\sqrt{nk_n})
\end{equation}
for some $\alpha \in(0,1) $. Recall that  $k_n=\log n$ is not large while $k_n=(\log n)^{1+\eps}$ is large.

We first observe that de-Poissonization cannot be applied as in the previous section. We lack smoothness as, for instance, $\mathbb{E}[\mathcal{L}_{\leq }(\mathcal{S}_{i_1,0,0,\dots,0})]=i_1\neq \mathcal{O}(\sqrt{\sum i_\ell})$.
The strategy is to apply Theorem \ref{Th:BorneInf_L} with
$$
x=nk_n,\qquad t=n,\qquad \lambda_n \approx \frac{1}{n}.
$$
(The exact value of $\lambda_n$  will be different for the proofs of the lower and upper bounds.)

\noindent{\bf Proof of the upper bound of \eqref{eq:Theta_leq} for large $(k_n)$.}

Choose $\alpha$ such that $n^2k_n\exp(-k_n^{\alpha})= \mathrm{o}(\sqrt{nk_n})$.
Put
$$
\lambda_n=\frac{1}{n}+\frac{\delta_n}{n},\qquad \text{ with } \delta_n=k_n^{-(1-\alpha)/2}.
$$
Let $E_n^{\lambda_n}$ be the event
$$
E_n^{\lambda_n}=\left\{ \text{ there are at least $k_n$ points in each row of  }\Pi_{nk_n,n}^{(\lambda_n)} \right\}.
$$
The event $E_n$ occurs with large probability. Indeed,
\begin{align}
1-\mathbb{P}(E_n^{\lambda_n})&\leq n \mathbb{P}\left( \Poi(nk_n \lambda_n) \leq k_n\right)\notag \\
&\leq n \mathbb{P}\left( \Poi(nk_n   \lambda_n) \leq nk_n   \lambda_n +k_n- nk_n   \lambda_n\right)\notag\\
&\leq n \mathbb{P}\left( \Poi(nk_n  \lambda_n) \leq nk_n   \lambda_n -k_n\delta_n\right)\notag\\
&\leq n\exp\left(-\frac{k_n^2\delta_n^2}{4nk_n\lambda_n}\right) \leq n\exp\left(-\tfrac{1}{8}k_n\delta_n^2\right)=  n\exp\left(-\tfrac{1}{8}k_n^{\alpha}\right). \label{eq:E}
\end{align}
At the last line we used Lemma  \ref{lem:Chernov}.
The latter probability tends to $0$ as $(k_n)$ is large.
\begin{lem}\label{lem:Couplage}
Random sets $S_{k_n;n}$ and $\Pi_{nk_n,n}^{(\lambda_n)} $ can be defined on the same probability space in such a way that 
\begin{align}
\mathcal{L}_{\leq }(S_{k_n;n})&\leq \mathcal{L}_{\leq}(\Pi_{nk_n,n}^{(\lambda_n)} )+ nk_n (1-\mathbf{1}_{E_n^{\lambda_n}}).\label{eq:Couplage_leq}
\end{align}
\end{lem}
\begin{proof}[Proof of Lemma \ref{lem:Couplage}]
Draw a sample of $\Pi_{nk_n,n}^{(\lambda_n)} $ and let $\tilde{\Pi}_{nk_n,n}^{(\lambda_n)} $ be the subset of  $\Pi_{nk_n,n}^{(\lambda_n)}$ obtained by keeping only the $k_n$ leftmost points in each row. If $E_n^{\lambda_n}$ occurs then the relative orders of points in $\tilde{\Pi}_{nk_n,n}^{(\lambda_n)}$  corresponds to a uniform $k_n$-multiset permutation. If $E_n^{\lambda_n}$ does not hold we bound  $\mathcal{L}_{\leq }(S_{k_n;n})$ by the worst case $nk_n$.
\end{proof}

Taking expectations in \eqref{eq:Couplage_leq} and using the upper bound \eqref{eq:esperance_upper_leq} yields
\begin{align*}
\mathbb{E}[\mathcal{L}_{\leq }(S_{k_n;n})]&\leq  2\sqrt{nk_n( 1+\delta_n)}+  k_n(1+\delta_n) + n^2k_n \exp\left(-\tfrac{1}{8}k_n^{\alpha}\right),
\end{align*}
hence the upper bound in \eqref{eq:Theta_leq}.

\medskip


\noindent{\bf Proof of the lower bound of  \eqref{eq:Theta_leq} for large $(k_n)$.}
Choose now $\lambda_n=\frac{1}{n}(1-\delta_n)$ with $\delta_n=k_n^{-(1-\alpha)/2}$. Let $F_n$ be the event
$$
F_n^{\lambda_n}=\left\{ \text{ at most $k_n$ points in each row of  }\Pi_{nk_n,n}^{(\lambda_n)}\right\}.
$$
The event $F_n^{\lambda_n}$ occurs with large probability. Indeed
$$
1-\mathbb{P}(F_n^{\lambda_n})\leq n \mathbb{P}\left( \Poi(nk_n \lambda_n) \geq k_n\right) \leq n\exp\left(-\tfrac{1}{8}k_n^{\alpha}\right),
$$
which tends to zero. 
Random sets $S_{k_n;n}$ and $\Pi_{nk_n,n}^{(\lambda_n)}$ can be defined on the same probability space in such a way that
\begin{align*}
\mathcal{L}_{\leq }(S_{k_n;n})&\geq \mathcal{L}_{\leq}(\Pi_{nk_n,n}^{(\lambda_n)})\mathbf{1}_{F_n^{\lambda_n}}.
\end{align*}
Therefore
\begin{multline}
\mathbb{P}\left(\mathcal{L}_{\leq }(S_{k_n;n}) <  (2\sqrt{nk_n(1-\delta_n)}+ k_n(1-\delta_n))(1-\eps)\right) \leq\\
 \mathbb{P}\left( \mathcal{L}_{\leq}(\Pi_{nk_n,n}^{(\lambda_n)})<   (2\sqrt{nk_n(1-\delta_n)}+ k_n(1-\delta_n))(1-\eps)\right) 
+  \mathbb{P}\left(\text{not }F_n^{\lambda_n}\right).\notag
\end{multline}
and we conclude with \eqref{eq:lower_leq}.

\subsection{The gap between small and large $(k_n)$: Conclusion of the proof of Theorem \ref{Th:Large}}

After I circulated a preliminary version of this article, Valentin F\'eray came up with a simple argument for bridging the gap between small and large $(k_n)$. This allows to prove Theorem \ref{Th:Large} for an arbitrary sequence $(k_n)$, I reproduce his argument here with his permission.

\begin{lem}\label{lem:Valentin} Let $n,k,A$ be positive integers. Two random uniform multiset permutations $\widetilde{S}_{kA ;\lfloor n/A\rfloor}$ and $S_{k;n}$ can be built on the same probability space in such a way that
$$
\mathcal{L}_{\leq}\left(S_{k ; n}\right) 
\leq
\mathcal{L}_{\leq}\left(\widetilde{S}_{kA ;\lfloor n/A\rfloor}\right)
+ kA.
$$
\end{lem}

\begin{proof}[Proof of Lemma \ref{lem:Valentin} ]
Draw $S_{k;n}$ uniformly at random, the idea is to group all points of $S_{k;n}$ whose height is between $1$ and $A$, to group all points whose height is between $A+1$ and $2A$, and so on.

Formally, denote by $1\leq i_1<i_2<\dots<i_{kA\lfloor n/A\rfloor}$ the indices such that $1\leq i_\ell \leq \lfloor n/A\rfloor$ for every $\ell$ (see Fig.\ref{fig:Valentin}). For $1\leq \ell \leq kA \lfloor n/A\rfloor$ put
$$
\widetilde{S}(\ell)=\lceil S(i_\ell)/k \rceil.
$$

\begin{figure}
\begin{center}
\includegraphics[width=120mm]{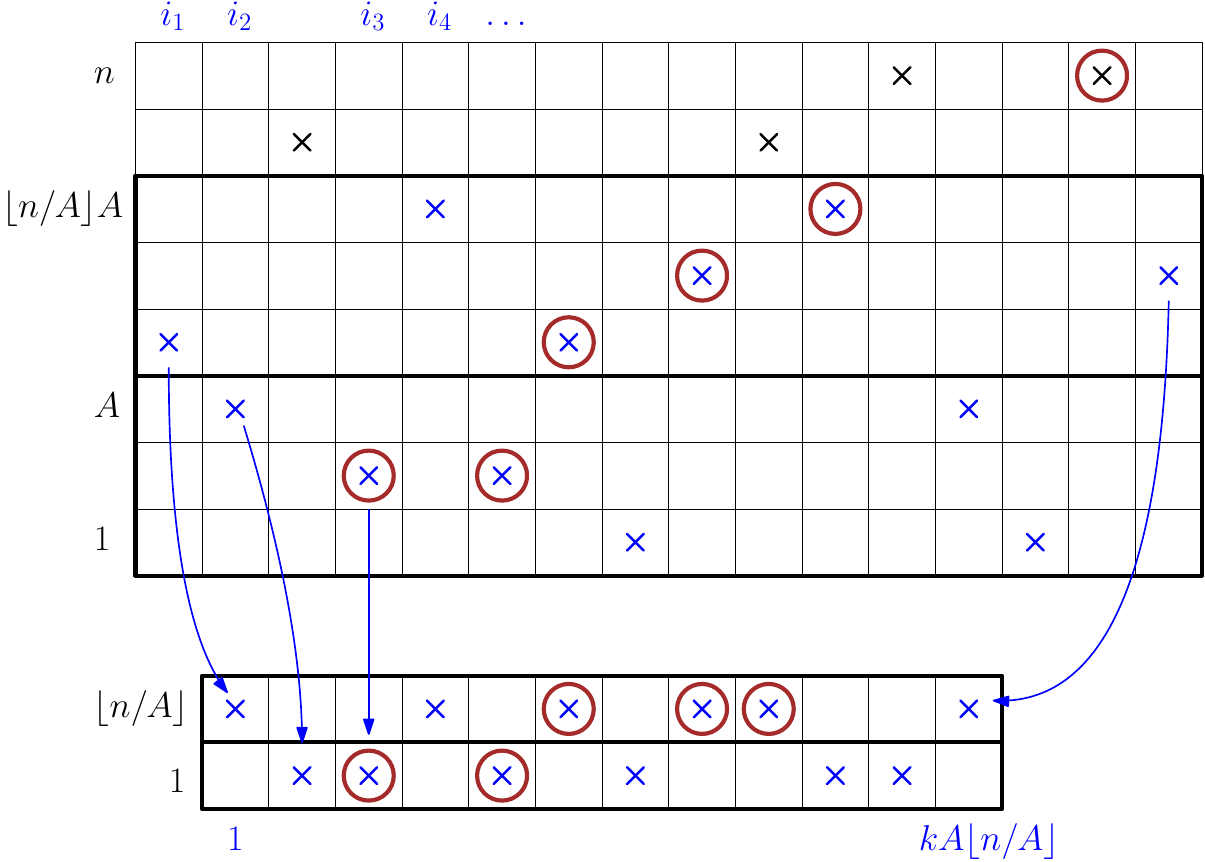} 
\end{center}
\caption{Illustration of the notation of Lemma \ref{lem:Valentin}. 
Top: the multiset permutation $S_{k;n}$. Bottom: the corresponding $\widetilde{S}$. The longest non-decreasing subsequence in $S_{k;n}$ (circled points) is mapped onto a non-decreasing subsequence in  $\widetilde{S}$, except one point with height $>A\lfloor n/A\rfloor$. }
\label{fig:Valentin}
\end{figure}

The word  $\widetilde{S}$ is a uniform $kA$-multiset permutation of size $ \lfloor n/A\rfloor$. A longest non-decreasing subsequence in $S$ is mapped onto a non-decreasing subsequence in $\widetilde{S}$, except maybe some points with height $> A\lfloor n/A\rfloor$ (there are no more than $kA$ such points). This shows the Lemma.
\end{proof}

We conclude the proof of Theorem \ref{Th:Large} by an estimation of $\mathbb{E}[\mathcal{L}_{\leq}\left(S_{k_n ; n}\right) ]$ in the case where there are infinitely many $k_n$'s  such that, say, $(\log n)^{3/4}\leq k_n \leq (\log n)^{5/4}$.
For the lower bound the job is already done by  Theorem \ref{Th:Strict} since
$$
\mathbb{E}[\mathcal{L}_{\leq}\left(S_{k_n ; n}\right) ]\geq \mathbb{E}[\mathcal{L}_{<}\left(S_{k_n ; n}\right) ]=2\sqrt{nk_n}-k_n+\mathrm{o}(\sqrt{nk_n}),
$$
which is of course also $2\sqrt{nk_n}+\mathrm{o}(nk_n)$ for this range of $(k_n)$. For the upper bound take $A=\lfloor \log n\rfloor$ in Lemma  \ref{lem:Valentin}:
\begin{equation}\label{eq:ReductionLarge}
\mathbb{E}[\mathcal{L}_{\leq}\left(S_{k_n ; n}\right) ]\leq \mathbb{E}[\mathcal{L}_{\leq}\left(S_{k_n \log n ;\lfloor n/\lfloor \log n\rfloor \rfloor}\right)]
+k_n\log n
\end{equation}
and we can apply the large case since
$$
(n/\log n)^2 k_n\log n \exp(-(k_n\log n)^{\alpha})=\mathrm{o}(k_n \log n \times \lfloor n/\lfloor \log n\rfloor \rfloor).
$$
Thus the right-hand side of \eqref{eq:ReductionLarge} is also $2\sqrt{nk_n}+\mathrm{o}(\sqrt{nk_n})$.

\section{Conclusion: Proof of Proposition \ref{prop:cvProba}}\label{Sec:Deviations}

In this short section we give the arguments needed to enhance estimates in expectation into convergences in probability. We have to prove that for every $\eps>0$
\begin{align*}
\mathbb{P}\left(L_{<}\left(S_{k_n ; n}\right) >(2\sqrt{nk_n}-k_n)(1+\eps)\right)&\to 0,\qquad 
\mathbb{P}\left(L_{<}\left(S_{k_n ; n}\right) <(2\sqrt{nk_n}-k_n)(1-\eps)\right)\to 0\\
\mathbb{P}\left(L_{\leq}\left(S_{k_n ; n}\right) >(2\sqrt{nk_n}+k_n)(1+\eps)\right)&\to 0,\qquad
\mathbb{P}\left(L_{\leq}\left(S_{k_n ; n}\right) <(2\sqrt{nk_n}+k_n)(1-\eps)\right)\to 0
\end{align*}
We only write the details for the top-left case, as the three other ones are almost identical.

The case where $(k_n)$ is small has been proved in Section \ref{Sec:Small} so it remains to prove the case where $(k_n)$ is large. We reuse the event  $E_n^{\lambda_n}$ introduced in Section \ref{Sec:DepoissonizationLarge}. 
\begin{align*}
\mathbb{P}\left(L_{<}\left(S_{k_n ; n}\right) >(2\sqrt{nk_n}-k_n)(1+\eps)\right)
\leq &\ \mathbb{P}\left(E_n^{\lambda_n} \text{ does not occur}\right)\\
+&\ \mathbb{P}\left(\mathcal{L}_{<}\left(\Pi^{(\lambda_n)}_{nk_n,n}\right)>(1+\delta_n)(2\sqrt{nk_n}-k_n)\frac{1+\eps}{1+\delta_n}\right)\\
\leq &\ n\exp\left(-\tfrac{1}{8}k_n^{\alpha}\right) \qquad \text{(recall \eqref{eq:E})}\\
+&\ \mathbb{P}\left(\mathcal{L}_{<}\left(\Pi^{(\lambda_n)}_{nk_n,n}\right)>(2\sqrt{nk_n(1+\delta_n)}-k_n(1+\delta_n))\frac{1+\eps}{1+\delta_n}\right)\\\leq &\ n\exp\left(-\tfrac{1}{8}k_n^{\alpha}\right) +  \exp(-\tilde{g}(\varepsilon/2)(\sqrt{nk_n}-k_n)),
\end{align*}
for large enough $n$ and for some positive $\tilde{g}$, using \eqref{eq:upper}. This tends to zero as desired.

The lower bound for $L_{<}(S_{k_n ; n})$ is proved in the same way. For the convergence of $L_{\leq}(S_{k_n ; n})$ we reuse the event $F_n^{\lambda_n}$ with $\lambda_n=\frac{1}{n}(1+\log(n))$.


\appendix
\section{Useful tail inequalities}

We collect here for convenience some (non-optimal) tail inequalities. 

\begin{lem}[(See Chap.2 in \cite{JansonChernov})]\label{lem:Chernov}
Let $\Poi(\lambda)$ be a Poisson random variable with mean $\lambda$. 
\begin{itemize}
\item[(i)]\label{item:Poisson_upper} For every $0<A\leq 3\lambda$,
$$
\mathbb{P}\left( \Poi(\lambda) \leq \lambda-A\right) \leq \exp(- A^2/4\lambda).
$$
\item[(ii)]\label{item:Poisson_lower} For every $A>0$,
$$
\mathbb{P}\left( \Poi(\lambda) \geq \lambda+A\right) \leq \exp(- A^2/4\lambda).
$$
\end{itemize}
\end{lem}

\begin{lem}[Th.2.1 in \cite{JansonChernov}]\label{lem:ChernovB} Let $\Bin(n,p)$ be a Binomial random variable with parameters $(n,p)$.
For $0<\varepsilon<1$,
\begin{align*}
\mathbb{P}( \Bin(n,p)\leq np-\varepsilon np) &\leq \exp\left(-\eps^2 np/2\right),\\
\mathbb{P}( \Bin(n,p) \geq np + \varepsilon np) &\leq  \exp\left(-\eps^2 np/3\right).
\end{align*}
\end{lem}

\begin{lem}\label{lem:ChernovG} Fix $\alpha\in(0,1)$ and let $\mathcal{G}_1^{(\alpha)}, \dots ,\mathcal{G}_k^{(\alpha)}$ be i.i.d. random variables with distribution
$\mathrm{Geometric}_{\geq 0}(1-\alpha)$. Then $\mathbb{E}[\mathcal{G}_1^{(\alpha)}]=\frac{\alpha}{1-\alpha}$ and for every $0<\varepsilon<1$,
\begin{align*}
\mathbb{P}\left( \mathcal{G}_1^{(\alpha)}+\dots + \mathcal{G}_k^{(\alpha)}\geq (1+\eps)k\frac{\alpha}{1-\alpha}\right) &\leq 
\exp\left(-\eps^2 k \alpha/20 \right),\\
\mathbb{P}\left( \mathcal{G}_1^{(\alpha)}+\dots + \mathcal{G}_k^{(\alpha)}\leq (1-\eps)k\frac{\alpha}{1-\alpha}\right) &\leq 
\exp\left(-\eps^2 k \alpha/20 \right).
\end{align*}
\end{lem}

\begin{proof}[Proof of Lemma \ref{lem:ChernovG}]
We will use the two inequalities
$$
\exp(z)\leq 1+z+z^2 \text{ for }|z|<1,\qquad \frac{1}{1-u}\leq \exp(u+u^2) \text{ for }|u|<1/2.
$$
Fix $\lambda$ such that $|\lambda| <\min\set{1,(1-\alpha)/4\alpha}$ so that $\frac{\alpha}{1-\alpha} |\lambda+\lambda^2|<1/2$:
\begin{align*}
\mathbb{E}[e^{\lambda(\mathcal{G}_1^{(\alpha)}-\frac{\alpha}{1-\alpha})} ]
&=\frac{(1-\alpha) }{1-\alpha e^{\lambda}}e^{-\lambda \frac{\alpha}{1-\alpha}}=\frac{1 }{1-\frac{\alpha}{1-\alpha} (e^{\lambda}-1)}e^{-\lambda \frac{\alpha}{1-\alpha}}\\
&\leq \frac{1 }{1-\frac{\alpha}{1-\alpha} (\lambda+\lambda^2)}e^{-\lambda \frac{\alpha}{1-\alpha}}\\
&\leq \exp\left(\frac{\alpha}{1-\alpha} (\lambda+\lambda^2)+\left(\frac{\alpha}{1-\alpha}\right)^2 (\lambda+\lambda^2)^2 -\lambda \frac{\alpha}{1-\alpha}\right)\\
&\leq \exp\left(\frac{\alpha}{(1-\alpha)^2}\lambda^2  \left(2+\lambda^2+2\lambda\right)\right)
\leq \exp\left( 5\lambda^2 \frac{\alpha}{(1-\alpha)^2}\right).
\end{align*}
Thus, for every $|\lambda|<\frac{1}{\beta}:=\min\set{1,(1-\alpha)/4\alpha}$ it holds that
$
\mathbb{E}[e^{\lambda(\sum_{i=1}^k\mathcal{G}_i^{(\alpha)}-k\frac{\alpha}{1-\alpha})} ]\leq \exp\left(\frac{\nu^2\lambda^2}{2}\right)
$
where $\nu^2:=10k\alpha/(1-\alpha)^2$.

This says that for every $k\geq 1$ the random variable $\mathcal{G}_1^{(\alpha)}+\dots + \mathcal{G}_k^{(\alpha)}$ is subexponential and the Chernov method applies (use \emph{e.g.} \cite[Prop.2.9]{Wainwright} with $t=\eps k\alpha/(1-\alpha)$):
\begin{align*}
\mathbb{P}\left( \mathcal{G}_1^{(\alpha)}+\dots + \mathcal{G}_k^{(\alpha)}\geq (1+\eps)k\frac{\alpha}{1-\alpha}\right) 
&\leq \exp\left(-\frac{t^2}{2\nu^2} \right)=\exp\left(-\eps^2 k^2 \frac{\alpha^2}{(1-\alpha)^2}\times \frac{(1-\alpha)^2}{2\times 10 k\alpha} \right)\\
&=\exp\left(-\eps^2 k \alpha/20 \right),
\end{align*}
as long as 
$$
\eps k\frac{\alpha}{1-\alpha}\leq \frac{\nu^2}{\beta}= 10k\frac{\alpha}{(1-\alpha)^2}\min\set{1,(1-\alpha)/4\alpha}
$$
which is always the case if $\eps <1$. The similar inequality holds for the left-tail bound (see \cite[Prop.2.9]{Wainwright} again).
\end{proof}


\section{An invariance property for the M/M/1 queue }\label{Sec:Appendix}

To conclude we state and prove the very simple property of the recurrent M/M/1 queue which allows to prove stationarity in Lemma \ref{lem:Stationnaire_leq}. It is very close to Burke's property of the discrete HAD process \cite{FerrariMartin}.

Let $\beta>\lambda>0$ be fixed parameters. Consider two independent homogeneous Poisson Point Process (PPP) $\Pi_\nearrow,\Pi_\searrow$ over $(0,+\infty)$ with respective intensities $\lambda,\beta$. Let $(H_y)_{y\geq 0}$ be the queue whose '+1' steps (customer arrivals) are given by $\Pi_\nearrow$ and '-1' steps (service times) are given by $\Pi_\searrow$ and whose initial distribution $H_0$ is drawn (independently from $\Pi_\nearrow,\Pi_\searrow$ ) according to a $\mathrm{Geometric}_{\geq 0}(1-\beta^\star)$ with $\beta^\star=\lambda/\beta$.

\begin{figure}
\begin{center}
\includegraphics[width=120mm]{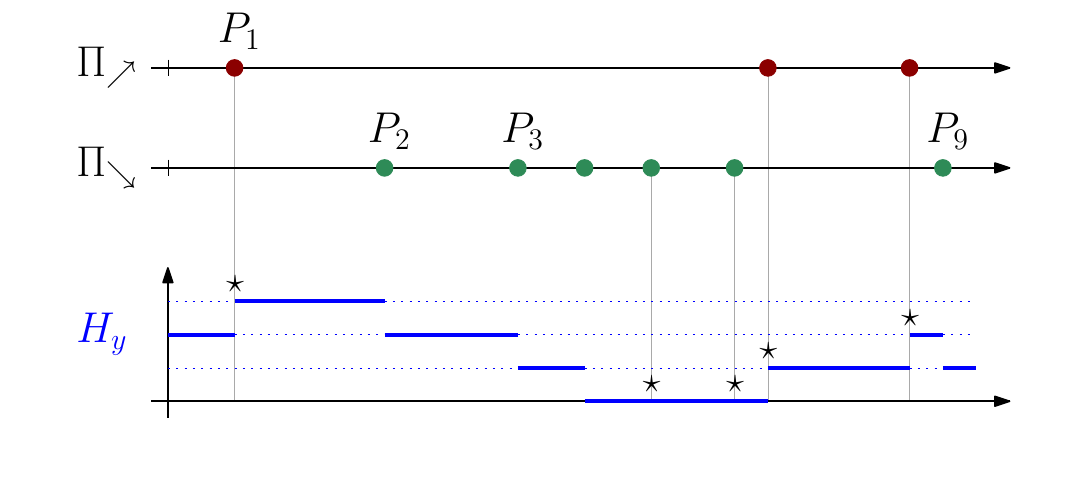} 
\end{center}
\caption{Notation of Lemma \ref{lem:Burke}. Points of $\overline{\Pi}$ are depicted with $\star$'s. }
\label{fig:H}
\end{figure}

Let $\Pi_0$ be the point process given by \emph{unused} service times:
$$
\Pi_0=\set{y\in \Pi_\searrow\text{ such that } H_y=0}.
$$
\begin{lem}\label{lem:Burke}
The process $\overline{\Pi}:=\Pi_\nearrow\cup \Pi_0$ is a homogeneous PPP with intensity $\beta$.
\end{lem}

\begin{proof} (The reader is invited to look at Fig.\ref{fig:H} for notation.)\\
The point process $\Pi_\nearrow\cup \Pi_\searrow$ is a homogeneous PPP with intensity $\lambda+\beta$, independent from $H_0$. We claim that $\overline{\Pi}$ is a subset of $\Pi_\nearrow\cup \Pi_\searrow$ where each point in  $\Pi_\nearrow\cup \Pi_\searrow$ is taken independently with probability $\beta/(\lambda+\beta)$, it is therefore a homogeneous  PPP with intensity $\beta$. 

We need a few notation in order to prove the claim. Set $P_0=0$ and for $i\geq 1$ let $P_i$ be the $i$-th point of $\Pi_\nearrow\cup \Pi_\searrow$ and let $(\tilde{H}_i)_{i\geq 0}$ be the discrete-time embedded chain associated to $H$, \emph{i.e.} $\tilde{H}_i=H_{P_i}$ for every $i$.

We will prove by induction that for every $i\geq 1$:
\begin{itemize}
\item the points $P_i$ belongs to $\overline{\Pi}$ with probability $\beta/(\lambda+\beta)$ independently from the events $\{P_1\in \overline{\Pi}\},\dots,\{P_{i-1}\in\overline{\Pi}\}$;
\item  $\tilde{H}_i$ is independent from $\{P_1\in \overline{\Pi}\},\dots,\{P_{i}\in\overline{\Pi}\}$ and is a $\mathrm{Geometric}_{\geq 0}(1-\beta^\star)$. 
\end{itemize}
This implies the claim and proves the Lemma. For the base case:
\begin{align*}
\mathbb{P}(P_1\in \overline{\Pi},\tilde{H}_1=k)&= \mathbb{P}(P_1\in \Pi_\nearrow, \tilde{H}_0=k-1 )\mathbf{1}_{k\geq 1}+\mathbb{P}(P_1\in \Pi_\searrow,\tilde{H}_0=0)\mathbf{1}_{k=0},\\
&=\frac{\lambda}{\lambda+\beta}\times (1-\beta^\star)(\beta^\star)^{k-1}\mathbf{1}_{k\geq 1}+ \frac{\beta}{\lambda+\beta}\times (1-\beta^\star)
\mathbf{1}_{k=0},\\
&=\frac{\beta}{\lambda+\beta}\times (1-\beta^\star)(\beta^\star)^{k}\qquad \text{(recall $\beta\beta^\star=\lambda$).}
\end{align*}
More generally let $E_j$ be one of the two events $P_j\in \overline{\Pi}/P_j\notin \overline{\Pi}$.
\begin{align*}
\mathbb{P}(P_i\in \overline{\Pi},\tilde{H}_i=k\ |\ E_1,\dots,E_{i-1})= &\ \mathbb{P}(P_i\in \Pi_\nearrow, \tilde{H}_{i-1}=k-1 \ |\ E_1,\dots,E_{i-1} )\mathbf{1}_{k\geq 1}\\
+&\ \mathbb{P}(P_i\in \Pi_\searrow,\tilde{H}_{i-1}=0\ |\ E_1,\dots,E_{i-1})\mathbf{1}_{k=0},\\
=&\ \mathbb{P}(P_i\in \Pi_\nearrow, \tilde{H}_{i-1}=k-1)\mathbf{1}_{k\geq 1}\\
+&\ \mathbb{P}(P_i\in \Pi_\searrow,\tilde{H}_{i-1}=0)\mathbf{1}_{k=0},\qquad \text{(by induction hypothesis).}\\
=&\ \frac{\beta}{\lambda+\beta}\times (1-\beta^\star)(\beta^\star)^{k}.
\end{align*}

\end{proof}

\paragraph{Acknowledgements.} 
This work started as a collaboration with Anne-Laure Basdevant, I would like to thank her very warmly. 
I am also extremely indebted to Valentin F\'eray for Lemma \ref{lem:Valentin} and for having enlightened me on the links with \cite{Biane}.
Finally, thanks to the authors of \cite{ContinuouslyIncreasing} for their stimulating paper and to anonymous referees for their careful readings.

\bibliographystyle{alpha}
\bibliography{BiblioMultiset}

\vfill
\noindent \textsc{Lucas Gerin} \verb|gerin@cmap.polytechnique.fr|\\
\textsc{Cmap, Cnrs}, \'Ecole Polytechnique,\\
Institut Polytechnique de Paris,\\
Route de Saclay,\\
91120 Palaiseau Cedex (France).

\end{document}